%% file: main.tex
\documentclass[final]{siamltex}
\usepackage{amsmath,amsfonts,amssymb,mathtools}
\usepackage{latexsym}
\usepackage{color}
\usepackage[table, xcdraw, dvipsnames]{xcolor}
\usepackage[hidelinks]{hyperref}
\hypersetup{
    colorlinks=true,
    linkcolor={red!50!black},
    citecolor={blue!50!black},
    urlcolor={blue!80!black}
}
\usepackage{mathtools} 

\newcommand{\R}{\mathbb R}
\newcommand{\bA}{\mathbf A}

\newcommand{\bI}{\mathbf I}

\newcommand{\bP}{\mathbf P}

\newcommand{\bV}{\mathbf V}

\newcommand{\blf}{\mathbf f}
\newcommand{\bn}{\mathbf n}
\newcommand{\be}{\mathbf e}

\newcommand{\bu}{\mathbf u}
\newcommand{\bU}{\mathbf U}
\newcommand{\bv}{\mathbf v}

\newcommand{\T}{\mathcal T}

\newcommand{\divG}{{\mathop{\,\rm div}}_{\Gamma}}
\newcommand{\gradG}{\nabla_{\Gamma}}
\newcommand{\nablaG}{\nabla_{\Gamma}}

\newcommand{\cT}{\mathcal T}

\newcommand{\OGamma}{\Omega^\Gamma_h}

\newcommand{\tr}{{\rm tr}}

\newcommand{\lPlane}{L_{\vect x_0}}

\newtheorem{assumption}{Assumption}[section]

\newtheorem{remark}{Remark}[section]

\newcommand{\vect}[1]{\boldsymbol{\mathbf{#1}}}

\newcommand*\diff{\mathop{}\!\mathrm{d}}

\newcommand{\sphere}{{\Gamma_{\text{sph}}}}
\newcommand{\tor}{{\Gamma_{\text{tor}}}}
\newcommand{\lsphere}{{\phi_{\text{sph}}}}
\newcommand{\ltor}{{\phi_{\text{tor}}}}

\newcommand{\HOne}{{H^1}}
\newcommand{\LTwo}{{L^2}}
\newcommand{\LTwoSpace}[1][\Gamma]{{L^2\left({#1}\right)}}

\usepackage{subcaption}
\usepackage{graphicx}
\graphicspath{{img/}} 
\usepackage{wrapfig}
\newcommand{\includegraphicsw}[2][1.]{\includegraphics[width=#1\linewidth]{#2}}
\usepackage{multirow}
\usepackage{hhline}
\usepackage{float} 

\begin{document}
\noindent\footnotesize{Published in \href{http://www.ams.org/mcom/}{Mathematics of Computation}, DOI:~\href{https://doi.org/10.1090/mcom/3551 }{10.1090/mcom/3551}}	
\title{Inf-sup stability of the trace $\vect P_2$--$P_1$ Taylor--Hood elements for surface PDE\MakeLowercase{s}}
\author{
Maxim A. Olshanskii\thanks{Department of Mathematics, University of Houston, Houston, Texas 77204 (molshan@math.uh.edu).}
\and Arnold Reusken\thanks{Institut f\"ur Geometrie und Praktische  Mathematik, RWTH-Aachen University, D-52056 Aachen, Germany (reusken@igpm.rwth-aachen.de).}
\and
Alexander Zhiliakov\thanks{Department of Mathematics, University of Houston, Houston, Texas 77204 (alex@math.uh.edu).}
}
\maketitle

\begin{abstract}
	The paper studies a geometrically unfitted finite element method (FEM), known as trace FEM or cut FEM, for the numerical solution of the Stokes system posed on a closed smooth surface. A trace FEM based on standard Taylor--Hood (continuous $\vect P_2$--$P_1$) bulk elements is proposed. A so-called volume normal derivative stabilization, known from the literature on trace FEM, is an essential ingredient of this method. The key result proved in the paper is an inf-sup stability of the trace $\vect P_2$--$P_1$ finite element pair, with the stability constant uniformly bounded with respect to the discretization parameter and the position of the surface in the bulk mesh. Optimal order convergence of a consistent variant of the finite element method follows from this new stability result and interpolation properties of the trace FEM. Properties of the method are illustrated with numerical examples.
\end{abstract}

\section{Introduction}

Surface fluid equations arise in continuum models of thin fluidic layers such as liquid films and plasma membranes.
The Euler and the Navier--Stokes equations posed on manifolds is also a classical topic of analysis~\cite{ebin1970groups,Temam88,taylor1992analysis,arnold1999topological,mitrea2001navier}.
The literature on numerical analysis or numerical simulations of fluid systems  on manifolds, however,  is still  rather scarce; see  \cite{nitschke2012finite,barrett2014stable,reuther2015interplay,reusken2018stream,reuther2018solving,fries2018higher,olshanskii2018finite,olshanskii2019penalty} for recent contributions. Among those papers only \cite{olshanskii2018finite,olshanskii2019penalty} addressed an unfitted finite element method for the surface Stokes and the surface Navier--Stokes systems.
The choice of the geometrically  unfitted discretization (instead of the fitted surface FEM  based on direct triangulation of surfaces)
is motivated by the numerical modelling of  deformable material interfaces. For interfaces featuring lateral fluidity  this leads to systems of PDEs posed on evolving surfaces $\Gamma(t)$~\cite{arroyo2009,Gigaetal,Jankuhn1}, for which \emph{unfitted} discretizations have certain attractive properties concerning flexibility (no remeshing) and robustness (w.r.t.   handling of strong deformations and topological singularities). The Stokes problem on a steady surface arises as an auxiliary problem in such simulations, if one splits the system into (coupled) equations for radial and tangential motions~\cite{Jankuhn1}. In the literature such  unfitted methods are  trace FEM~\cite{olshanskii2016trace} or cut FEM~\cite{cutFEM}.

The $\vect P_2$--$P_1$ continuous Taylor--Hood element is one of the most popular FE pairs  for incompressible fluid flow problems. Surface variants of this pair have been used for surface Navier-Stokes equations in the recent papers \cite{reuther2018solving,fries2018higher}. In those papers the fitted surface FE approach is used. There is no literature in which surface variants of the  Taylor--Hood elements are studied in the context of unfitted discretizations. Furthermore, there is no literature in which rigorous stability or error analysis of surface variants of Taylor--Hood elements is presented.

In this paper we propose a surface variant of the Taylor--Hood element for an unfitted discretization of the surface Stokes problem. It turns out that a particular stabilization technique is essential, cf. below. A second main topic of the paper is a rigorous analysis of this method. We show that the $\vect P_2$--$P_1$ trace FEM that we propose is stable and has optimal order convergence in the surface $H^1$ and $L^2$ norms for velocity and the surface $L^2$-norm for pressure. The key result proved in section~\ref{sec:analysis} is the uniform inf-sup stability property for the trace spaces of   $\vect P_2$--$P_1$ elements. Hence, this paper contains the first optimal rigorous error bounds for a surface Stokes problem discretized using a surface variant of the Taylor--Hood pair.

It is standard for trace/cut FEM to add volumetric consistent stabilization terms  to improve algebraic properties of  the algebraic system, which results from the natural choice of basis functions corresponding to the bulk nodal basis. Here we use the volume normal stabilization~\cite{burman2016cutb,grande2016higher} for this purpose. It turns out that this stabilization (for pressure) is also crucial for the central inf-sup stability result. Numerical results demonstrate that our analysis is sharp in the sense that the discretization without this stabilization is not inf-sup stable.

The remainder of the paper is organized as follows. In section~\ref{sec:model} we collect necessary notations of tangential differential calculus and recall the mathematical model. Section~\ref{sec:FE} introduces the finite element method. The main theoretical result of the paper on  stability of trace $\vect P_2$--$P_1$ element is proved in section~\ref{sec:analysis}; see Theorem~\ref{lemma2}.  Section~\ref{sec:error} proceeds with the error analysis of the finite element method.
Results of numerical experiments, which illustrate relevant properties of the proposed discretization method, are presented in Section~\ref{sec:num}. Conclusions are given in the closing section~\ref{sec:concl}.

\section{Mathematical model}\label{sec:model}

In this section we recall the system of surface Stokes equations, which models the slow tangential motion of a surface fluid in a state of geometric equilibrium. For the purpose of further numerical analysis, it is convenient to formulate the problem in terms of tangential calculus. For derivations and further properties of fluid equations on manifolds see \cite{taylor1992analysis,arnold1999topological,mitrea2001navier} (for equations  written in intrinsic variables) and \cite{arroyo2009,Gigaetal,Jankuhn1} (for formulations using tangential calculus).

We consider a closed smooth surface $\Gamma\subset\mathbb{R}^3$ with the outward-pointing normal field $\bn$ and a (sufficiently small) three-dimensional neighborhood $\mathcal{O}(\Gamma)$.  For a scalar function $p:\, \Gamma \to \mathbb{R}$ or a vector function $\bu:\, \Gamma \to \mathbb{R}^3$  we assume any smooth extension $p^e:\,\mathcal{O}(\Gamma)\to\mathbb{R}$, $\bu^e:\,\mathcal{O}(\Gamma)\to\mathbb{R}^3$ of $p$ and $\bu$ from $\Gamma$ to its neighborhood $\mathcal{O}(\Gamma)$. For example, one can think of extending $p$ and $\bu$  with constant values  along the normal.
The surface gradient and covariant derivatives on $\Gamma$ are then defined as $\nablaG p \coloneqq \bP\nabla p^e$ and  $\nabla_\Gamma \bu\coloneqq  \bP \nabla \bu^e \bP$, with $\bP \coloneqq \vect I - \bn \bn^T$ the orthogonal projection onto the tangential plane (at $\vect x \in \Gamma$). The definitions  of surface gradient and covariant derivatives are  independent of a particular smooth extension of $p$ and $\bu$ off $\Gamma$.
The surface rate-of-strain tensor \cite{GurtinMurdoch75} on $\Gamma$ is given by
\begin{equation} \label{strain}
	E_s(\bu)\coloneqq  \frac12 \bP \big(\nabla \bu^e + (\nabla \bu^e)^T\big)\bP = \frac12(\nabla_\Gamma \bu + \nabla_\Gamma \bu^T).
\end{equation}
The surface divergence operators for a vector $\bv: \Gamma \to \R^3$ and
a tensor $\bA: \Gamma \to \mathbb{R}^{3\times 3}$ are defined as
\[
 \divG \bv  \coloneqq  \tr (\gradG \bv), \qquad
 \divG \bA  \coloneqq  \left( \divG (\be_1^T \bA),\,
               \divG (\be_2^T \bA),\,
               \divG (\be_3^T \bA)\right)^T,
               \]
with $\be_i$ the $i$th basis vector in $\R^3$.

The surface Stokes problem reads: For a given tangential force vector $\mathbf{f} \in L^2(\Gamma)^3$, i.e. $\mathbf{f}\cdot\bn=0$ holds, and
source term $g\in L^2(\Gamma)$, with $\int_\Gamma g\diff{s}=0$, find a tangential velocity field $\bu:\, \Gamma \to \R^3$, $\bu\cdot\bn =0$, and a surface fluid pressure $p:\, \Gamma \to \R$  such that
\begin{align} 
  - 2\,\bP \divG (E_s(\bu))+\alpha \bu +\nabla_\Gamma p &=  \blf \quad \text{on}~\Gamma,  \label{strongform-1} \\
  \divG \bu & =g \quad \text{on}~\Gamma, \label{strongform-2}
\end{align}
where $\alpha\ge0$  is a real parameter. The steady surface  Stokes problem corresponds to $\alpha=0$, while $\alpha>0$ leads to a generalized surface Stokes problem, which results from an implicit time integration applied to the time dependent equations. The body force $\blf$ models  exterior forces, such as a gravity force, and tangential stresses exerted by an ambient medium.
The  source term $g$  is non-zero, for example, if  \eqref{strongform-1}--\eqref{strongform-2} is used as an auxiliary problem for the modeling of evolving fluidic interfaces. In that case, the inextensibility condition reads $\divG \bu_T =-u_N\kappa$, where $\kappa$ is the mean curvature
and $u_N$ is the normal component of the velocity. Here and further in the paper,
we use the  decomposition of a general vector field into tangential and normal components:
\begin{equation}\label{u_T_N}
\bu = \bu_T + u_N\bn,\quad \bu_T\cdot\bn=0.
\end{equation}
For the derivation of the Navier--Stokes equations  for evolving fluidic interfaces see, e.g., \cite{Jankuhn1}.

As common for models of incompressible fluids, the pressure field is defined up to a hydrostatic mode. For $\alpha=0$
all tangentially rigid surface fluid motions,  i.e. satisfying  $E_s(\bu) = \vect 0$,
are in the kernel of the differential operators at the left-hand side of eq.~\eqref{strongform-1}.
Integration by parts implies the consistency condition for the right-hand side of eq.~\eqref{strongform-1}:
 \begin{equation}\label{constr}
 \int_\Gamma \blf \cdot \bv\diff{s}=0\quad\text{for all smooth tangential vector fields}~~\bv~~ \text{s.t.}~~ E_s(\bv)=\mathbf{0}.
 \end{equation}
This condition  is necessary for the well-posedness of problem \eqref{strongform-1}--\eqref{strongform-2} when $\alpha=0$.
In the literature a tangential vector field $\bv$ on $\Gamma$ satisfying
$E_s(\bv)=\mathbf{0}$ is known as a Killing vector field. For a smooth two-dimensional Riemannian manifold, Killing vector fields form a Lie algebra of dimension  at most 3 (cf., e.g., {Proposition III.6.5 in \cite{sakai1996riemannian}}) and the corresponding subspace plays an important role in the analysis of the surface fluid equations, cf. \cite{olshanskii2019penalty}.
It is reasonable to assume (see \cite[Remark~2.1]{olshanskii2018finite}) that either no non-trivial Killing vector field exists on
$\Gamma$ or $\alpha>0$. For the purpose of this paper, which focuses on stability properties  of certain surface finite elements,  we assume $\alpha=1$. The results that are obtained also hold (with minor
modifications) for the case $\alpha=0$, if $\Gamma$ is such that there is no non-trivial Killing vector field.
If a non-trivial Killing vector field is present  and $\alpha=0$, then an additional effort is needed as, for example, discussed in \cite{bonito2019divergence}, where an $\epsilon$-regularization of a finite element method is introduced to handle the kernel.

For the weak formulation of the surface Stokes problem \eqref{strongform-1}--\eqref{strongform-2},
we need the vector Sobolev space $\bV\coloneqq  H^1(\Gamma)^3$ equipped with the norm
\begin{equation} \label{H1norm}
	\|\bv\|_{1}\coloneqq \left(\|\bv\|^2_{L^2(\Gamma)} + \|(\nabla\bv^e)\bP\|^2_{L^2(\Gamma)}\right)^{\frac12}
\end{equation}
and its subspace of tangential vector fields
\begin{equation}   \label{defVT}
 \bV_T\coloneqq  \{\, \bv \in \bV\::\: \bv\cdot \bn =0\,\}.
\end{equation}
For $\bv \in \bV$ we will use the orthogonal decomposition into tangential and normal parts
as in \eqref{u_T_N}.
We define $L_0^2(\Gamma)\coloneqq \{\,p \in L^2(\Gamma)\::\:\int_\Gamma p \diff{s} = 0\,\}$.

Consider the continuous bilinear forms (with $\vect A : \vect B \coloneqq {\rm tr}\big(\vect A \vect B^T\big)$ for  $\vect A, \vect B\in\mathbb{R}^{3\times3}$)
\begin{align}
a(\bu,\bv) &\coloneqq \int_\Gamma (2\,E_s(\bu):E_s(\bv)+\bu\cdot\bv) \diff{s}, \quad \bu,\bv \in \bV, \label{defblfa} \\
b_T(\bv,p) &\coloneqq -\int_\Gamma p\,\divG \bv_T \diff{s},  \quad \bv \in \bV, ~p \in L^2(\Gamma). \label{defblfb}
\end{align}
Note that in the definition of $b_T(\bv,p)$ only the {tangential} component of $\bv$ is used, i.e., $b_T(\bv,p)=b_T(\bv_T,p)$ for all $\bv \in \bV$, $p\in L^2(\Gamma)$. This property motivates the notation $b_T(\cdot,\cdot)$ instead of $b(\cdot,\cdot)$.
If $p$ is from $H^1(\Gamma)$, then integration by parts yields
\begin{equation}\label{Bform}
b_T(\bv,p)=\int_\Gamma  \bv_T\cdot \gradG p \diff{s} = \int_\Gamma  \bv \cdot \gradG p \diff{s}.
\end{equation}

The weak  formulation of the surface Stokes problem \eqref{strongform-1}--\eqref{strongform-2}
reads: Find 
$(\bu_T,p) \in \bV_T \times L_0^2(\Gamma)$ 
such that
\begin{align}
	a(\bu_T,\bv_T) +b_T(\bv_T,p) &=(\blf,\bv_T) \quad \text{for all}~~\bv_T \in \bV_T, \label{Stokesweak1_1} \\
	b_T(\bu_T,q) & = (-g,q) \quad \text{for all}~~q \in L^2(\Gamma). \label{Stokesweak1_2}
\end{align}
Here $(\cdot,\cdot)$ denotes the $L^2$ scalar product on $\Gamma$.
The following surface Korn   inequality and inf-sup property were derived in \cite[result (4.8) and Lemma~4.2]{Jankuhn1}:
Assuming $\Gamma$ is $C^2$ smooth and compact,  there exist $c_K >0$ and $c_0>0$ such that
 \begin{equation}  \label{korn1}
\|\bv_T\|_{L^2(\Gamma)}+ \|E_s(\bv_T)\|_{L^2(\Gamma)} \geq c_K \|\bv_T\|_{1} \quad \text{for all}~~\bv_T \in \bV_T,
 \end{equation}
and
\begin{equation} \label{infsup}
 \sup_{\bv_T\in{\bV_T}}\frac{b_T(\bv_T,p)}{\|\bv_T\|_{1}}  \geq c_0 \|p\|_{L^2(\Gamma)} \quad \text{for all}~~p\in L^2_0(\Gamma).
\end{equation}
The equations  \eqref{korn1} and \eqref{infsup} guarantee  the coercivity and inf-sup stability of the bilinear forms $a(\cdot,\cdot)$ and $b_T(\cdot,\cdot)$, respectively. This, in turn, implies the well-posedness of the weak formulation \eqref{Stokesweak1_1}--\eqref{Stokesweak1_2}. The unique solution of \eqref{Stokesweak1_1}--\eqref{Stokesweak1_2} is denoted by $(\bu_T^\ast, p^\ast)$. This solution has $H^2\times H^1$-regularity once $\blf$ and $g$ are suitably regular. We include a proof of this result below for a compact closed $C^2$ surface.
\begin{lemma}\label{Lemma:regularity} Assume $\Gamma\in C^2$ compact and closed, $\mathbf{f}\in L^2(\Gamma)^3$, $g\in H^1(\Gamma)$, then
 $\bu_T^\ast\in H^2(\Gamma)^3$, $p^\ast\in H^1(\Gamma)$ and $\|\bu^\ast\|_{H^2(\Gamma)}+\|p^\ast\|_{H^1(\Gamma)}\le C(\Gamma)(\|\mathbf{f}\|_{L^2(\Gamma)}+\|g\|_{H^1(\Gamma)})$.
\end{lemma}
\begin{proof} Since $\Gamma$ is closed (no boundary conditions) the proof splits into two steps: first we show an  $H^1$-regularity estimate for $p^\ast$ from a pressure-Laplace-Beltrami equation and next we derive velocity $H^2$-regularity from a Hodge--Laplace equation satisfied by $\bu_T^\ast$.
For the first step we use the identity (cf. \cite{reusken2018stream})
\begin{equation} \label{helprel} \begin{split}
 \int_\Gamma 2 E_s(\bu_T):E_s(\bv_T)\, ds & = \int_{\Gamma} {\rm curl}_\Gamma \bu_T\, {\rm curl}_\Gamma \bv_T \, ds +2 \int_{\Gamma} \divG \bu_T  \divG \bv_T \, ds  \\ & ~~ - 2\int_\Gamma K \bu_T \cdot \bv_T\, ds,
\end{split} \end{equation}
where $K$ is the Gaussian curvature and ${\rm curl}_\Gamma$ the scalar surface curl-operator.
We take $\bv_T=\nabla_\Gamma q$, for  $q\in H^2_\ast(\Gamma):=H^2(\Gamma) \cap L_0^2(\Gamma)$,  in \eqref{Stokesweak1_1} and use \eqref{helprel}, ${\rm curl}_\Gamma \nabla_\Gamma q=0$ and \eqref{Stokesweak1_2}, to get
\begin{equation}\label{weak_PP}
-(p^\ast,\Delta_\Gamma q)_{L^2(\Gamma)}=(\mathbf{g},\nablaG q)_{L^2(\Gamma)},\quad \text{with}~\mathbf{g}=\blf+2\nablaG g +(2K-1)\bu_T.
\end{equation}
  Thanks to the $H^2$-regularity of the Laplace--Beltrami problem  
the bilinear form $(p,q) \to -(p,\Delta_\Gamma q)_{L^2(\Gamma)}$ is infsup stable in $L^2(\Gamma)\times H^2_\ast(\Gamma)$. Furthermore, $(p,\Delta_\Gamma q)_{L^2(\Gamma)}=0$ for all $p \in L^2(\Gamma)$ implies $q=0$. Therefore, \eqref{weak_PP} is a well-posed (very) weak formulation of the pressure Laplace--Beltrami problem. Since $\mathbf{g}\in L^2(\Gamma)^3$ we have $p^\ast\in H^1(\Gamma)$ and the standard energy estimate implies the desired bound for the unique pressure solution:
$\|p^\ast\|_{H^1(\Gamma)}\le \|\mathbf{g}\|_{L^2(\Gamma)}\le C(\Gamma)(\|\mathbf{f}\|_{L^2(\Gamma)}+\|g\|_{H^1(\Gamma)})$.\\
We proceed to the second step and  employ (\cite{reusken2018stream}) the relation $ 2\,\bP \divG (E_s(\bu)) = \Delta_\Gamma^H+\gradG \divG + 2K\bu$,
with $\Delta_\Gamma^H$ the Hodge Laplacian.  Using this we rewrite (in suitable weak form) the first equation in the Stokes system as
$-\Delta_\Gamma^H\bu=\mathbf{g_u}:=\blf +(2K-1) \bu+\nablaG g-\nabla_\Gamma p$.  Now the standard elliptic regularity, cf. e.g. \cite[section 7.4]{Morrey},  implies $\|\bu\|_{H^2}\le C\|\mathbf{g_u}\|_{L^2(\Gamma)}$, and the desired bound for the velocity follows by suitably bounding each term in $\mathbf{g_u}$.
\end{proof}

For the discrete surface Stokes problem the situation is similar to the planar case in the following sense: While the coercivity of the finite element velocity  form  follows immediately from the analogous  property of the original formulation, the inf-sup stability of the $b$-form for a given pair of finite element spaces is a delicate question. Here we address this question for the unfitted (trace) variant of the $\vect P_2$--$P_1$ Taylor--Hood elements.
First we introduce the finite element discretization of~\eqref{Stokesweak1_1}--\eqref{Stokesweak1_2}.

\section{Finite element discretization}\label{sec:FE}

We apply an unfitted finite element method, the trace FEM~\cite{olshanskii2016trace}, for the discretization of \eqref{Stokesweak1_1}--\eqref{Stokesweak1_2}.
This method uses a surface-independent ambient (bulk) mesh of an immersed manifold to discretize a PDE.
To formulate the method, consider  a fixed polygonal domain  $\Omega \subset \R^3$ that strictly contains $\Gamma$.
Assume a family of \emph{shape regular} tetrahedral triangulations $\{\T_h\}_{h >0}$ of $\Omega$.
The subset of tetrahedra that have a nonzero intersection with $\Gamma$ is collected in the set
denoted by $\T_h^\Gamma$. Tetrahedra from $\T_h^\Gamma$ form our active computational mesh.
For $h_T \coloneqq \mbox{diam}(T)$ we denote $h \coloneqq \max_{T\in\T_h^{\Gamma}} h_T$.  In the numerical section we denote the typical meshsize of~$\T_h^\Gamma$ by~$h$.
For the analysis of the method,  we assume $\{\T_h^\Gamma\}_{h >0}$ to be quasi-uniform: $h/\min_{T\in\T_h^\Gamma} h_T\le C$, with a constant
$C$ independent of $h$.  Moreover, for the stability analysis, we shall need the following technical assumption on how $\T_h^\Gamma$ resolves $\Gamma$:
\begin{equation}\label{Ass1}
\forall~T\in\T_h^\Gamma:\quad  T\cap\Gamma~\text{is simply-connected}~~\text{and}~~ |\partial T\cap\Gamma|\le C\,h_T,
\end{equation}
with some $C$ independent of $T$.  The assumption on simply-connectivity of $T\cap\Gamma$ can be further relaxed to assuming that the number of connected components in $T\cap\Gamma$ is uniformly bounded, but we will  not pursue this technical improvement further.

The domain formed by all tetrahedra in $\T_h^\Gamma$ is denoted by $\OGamma \coloneqq \text{int}(\cup_{T \in \T_h^\Gamma} \overline{T})$.
On $\T_h^\Gamma$ we use standard finite element spaces of continuous functions, which are polynomials of degree $k$ on each tetrahedron.
These so-called \emph{bulk finite element spaces} are denoted by $V_h^k$,
\begin{equation} \label{fespace}
V_h^k=\{v\in C(\OGamma)\,:\, v\in P_k(T)~\forall~T\in\T_h^\Gamma\}.
\end{equation}
Our bulk velocity and pressure finite element spaces are Taylor--Hood elements on $\OGamma$:
\begin{equation} \label{TaylorHood}
\bU_h \coloneqq  (V_h^{k+1})^3, \quad Q_h \coloneqq  V_h^k\cap L^2_0(\Gamma), ~~k \geq 1. 
\end{equation}
In the trace finite element method formulated below, \emph{traces of functions from $\bU_h$ and $Q_h$ on $\Gamma$ are used to discretize the surface Stokes system}.
\begin{assumption} \label{ass1}
We assume that integrals over $\Gamma$ can be computed exactly, i.e. we do not consider geometry errors.
\end{assumption}

In practice $\Gamma$ has to be approximated by a (sufficiently accurate) approximation $\Gamma_h \approx \Gamma$ in such a way that integrals over $\Gamma_h$ can be computed accurately and efficiently, cf. Remark~\ref{Rem1}.
\begin{remark} \label{Rem1} \rm  For the $\vect P_{k+1}$--$P_k$ Taylor--Hood pair the optimal rate of convergence for velocity (in the $L^2$ norm) is  $O(h^{k+2})$. A piecewise planar approximation  $\Gamma_h\approx \Gamma$ leads to an  $O(h^2)$ geometric error and a suboptimal discretization error. To overcome this, for the trace FEM a general  \emph{higher order} technique, based on a parametric mapping of the domain $\OGamma$, has been developed \cite{grande2018analysis}. This approach can be directly applied to the Taylor--Hood spaces, cf. \cite{JankuhnStokes2019}. To avoid further technical issues related to the  analysis of the parametric mapping, in this paper we do not  study these isoparametric Taylor--Hood spaces. Instead we use Assumption~\ref{ass1} and analyze the spaces \eqref{TaylorHood}. Numerical results from \cite{JankuhnStokes2019} suggest  that   the stability properties of the trace spaces corresponding to the  pair \eqref{TaylorHood}, which is the focus of this paper, and of the
parametric variant of
this pair are essentially the same. We expect that the analysis of the current paper can be also extended to the isoparametric setting. {This topic will be treated in a forthcoming paper.}
\end{remark}
\\[1ex]
There are two important issues specifically related to the fact that we consider a \emph{surface} Stokes system. Firstly, the numerical treatment of the tangentiality condition $\bu \cdot\bn=0$ on $\Gamma$.  Enforcing  the  condition $\bu_h\cdot\bn=0$ on $\Gamma$ for polynomial functions  $\bu_h \in \bU_h$ is inconvenient and may lead to locking (only $\bu_h = \vect 0$ satisfies it). Following \cite{HANSBO2017298,hansbo2016analysis,Jankuhn1,reuther2018solving,olshanskii2018finite}  we add a penalty term to the weak formulation to enforce the tangential constraint weakly. The second issue is related to possible small cuts of tetrahedra from $\T_h^\Gamma$ by the surface. For the standard choice of finite element basis functions this may lead to poorly conditioned algebraic systems. The algebraic stability is recovered by adding certain volumetric terms to the finite element formulation.

Hence, the bilinear forms that we use in the discretization method contain terms related to algebraic stability and a penalty term. We introduce the  bilinear forms:
\begin{align}\begin{split}\label{forms}
	A_h(\bu, \bv) &\coloneqq
		\int_{\Gamma} (2E_s(\bu) : E_s(\bv) + \bu\cdot\bv + \tau\,u_N\,v_N)  \diff{s} \\ & ~~~+
		\rho_u \int_{\Omega_h^{\Gamma}} ([\nabla\bu]\,\bn)\cdot([\nabla\bv]\,\bn) \diff{\vect x},\\
s_h(p,q)& \coloneqq  \rho_p  \int_{\OGamma} (\bn \cdot \nabla p)  (\bn \cdot\nabla q) \,  \diff{\vect x},
\end{split}\end{align}
with the penalty parameter $\tau \geq 0$ and two stabilization parameters $\rho_p  \geq 0$ and $\rho_u \geq 0$. In  practice the (exact) normal $\bn$ used in the bilinear forms $A_h(\cdot,\cdot)$ and $s_h(\cdot,\cdot)$ is replaced by a sufficiently accurate approximation. 

The trace finite element method reads:
Find $(\bu_h, p_h) \in \bU_h \times Q_h$ such that
\begin{equation} \label{discrete}
 \begin{aligned}
  A_h(\bu_h,\bv_h) + b_T(\bv_h,p_h) & =(\blf,\bv_h) &\quad &\text{for all } \bv_h \in \bU_h, \\
  b_T(\bu_h,q_h)-s_h(p_h,q_h) & = (-g,q_h) &\quad &\text{for all }q_h \in Q_h.
 \end{aligned}
\end{equation}
We allow the following ranges of  parameters:
\begin{equation} \label{parameters}
\tau= c_\tau h^{-2},\quad \rho_p= c_p h, \quad \rho_u\in [c_u h,C_u h^{-1}].
\end{equation}
Here $h$ is the characteristic mesh size of the background tetrahedral mesh, while
$c_\tau$, $c_p$, $c_u$, $C_u$ are strictly positive constants  independent of $h$ and of how $\Gamma$ cuts through the background mesh. 
The volumetric term in the definition of $A_h$ is the so called \emph{volume normal derivative} stabilization,
first introduced  in \cite{burman2016cutb,grande2016higher} in the context of  trace FEM
for the scalar Laplace--Beltrami problem on a surface. The term vanishes for the strong solution $\bu$ of
 \eqref{strongform-1}--\eqref{strongform-2}, since one can always assume an extension of $\bu$ off the surface that is constant in normal direction, hence $\lbrack\nabla\bu\rbrack\,\bn = \vect 0$ on $\Omega_h^{\Gamma}$.
As  mentioned above, the  purpose of adding the integrals over the strip $\OGamma$ is to improve the condition number of the resulting algebraic systems.
  Consistency analysis yields the condition $\rho_u \leq C_u h^{-1}$.
While adding the volumetric stabilization for velocity is not essential for stability of the finite element method (only for algebraic conditioning), we shall see   that in the context of mixed trace FEM, the pressure volumetric term with $\rho_p \geq c_p h$ is crucial also
for good stability properties of the finite element discretization method.  In view of (optimal) consistency we take  $\rho_p= c_p h$.  Stability and error analysis suggest  $\tau= c_\tau h^{-2}$, the choice used throughout the paper.

\begin{remark}[Consistency] \label{remconsistent} \rm The discrete problem  \eqref{discrete} is \emph{not} consistent:  \eqref{discrete} is not satisfied with $(\bu_h, p_h)$ replaced by the true solution $(\bu^*, p^*)$ extended with constant values along the normal. Indeed, the velocity finite element space is not a subspace of $\bV_T$ and for the surface rate-of-strain tensor of  $\bv \in \bU_h$, we have
\[
E_s(\bv)=E_s(\bv_T)+ v_N \vect H,
\]
where the term containing the Weingarten map $\vect H \coloneqq \nabla \bn$ causes an inconsistency.  This inconsistency is removed  if instead of $A_h(\cdot,\cdot)$ one uses the bilinear form
\begin{equation} \label{Aconsistent} \begin{split}
 \tilde A_h(\bu, \bv) & \coloneqq
		\int_{\Gamma} \big( 2(E_s(\bu)-u_N \vect H) : (E_s(\bv)- v_N \vect H) + \bu\cdot\bv + \tau\,u_N\,v_N \big) \diff{s} \\ & ~~~ +
		\rho_u \int_{\Omega_h^{\Gamma}} ([\nabla\bu]\,\bn)\cdot([\nabla\bv]\,\bn) \diff{\vect x}.
		\end{split}
\end{equation}
 The consistency properties of the bilinear forms $A_h(\cdot,\cdot)$ in \eqref{forms} and $\tilde A_h(\cdot,\cdot)$ in \eqref{Aconsistent} are analyzed in~\cite{JankuhnVekorLaplace}.  In the analysis below we use $A_h(\cdot,\cdot)$, but all results also apply to $\tilde A_h(\cdot, \cdot)$, due to the equivalence  (for $h$ sufficiently small, implying $\tau$ sufficiently large):
 \begin{equation}\label{equivA}
\frac12 A_h(\bv,\bv)\leq  \tilde A_h(\bv,\bv) \leq 2 A_h(\bv,\bv)\quad \text{for all}~ \bv \in \bU_h.
 \end{equation}
\end{remark}

\section{Stability analysis of trace Taylor--Hood elements}\label{sec:analysis}

It is natural to study  the stability of the finite element method \eqref{discrete} using the following  problem-dependent norms in $\bU_h$ and $Q_h$:
\begin{equation}\label{norms}
\|\bv\|_A=A_h(\bv, \bv)^{\frac12},\quad\|q\|_h=\left(\|q\|^2_{L^2(\Gamma)}+s_h(q,q)\right)^\frac12.
\end{equation}
Functionals in \eqref{norms} indeed define norms on $\bU_h$ and $Q_h$ thanks to the  included volumetric terms (i.e. they define the norms not only on the trace spaces, but also on the spaces of bulk FE functions on $\OGamma$). In particular, it holds (cf. Lemma~7.4 from \cite{grande2018analysis}):
\begin{equation}\label{L2control}
h^{-\frac12}\|\bv\|_{L^2(\OGamma)}\le C\|\bv\|_A\quad\text{and}\quad h^{-\frac12}\|q\|_{L^2(\OGamma)}\le C\|q\|_{h}\quad\forall~\bv\in\bU_h,\,q\in Q_h,
\end{equation}
with a constant $C$ independent of $h$ and the position of $\Gamma$ in the mesh.

We immediately see that the forms $b_T(\cdot,\cdot)$ and $s_h(\cdot,\cdot)$ are continuous and
the form $A_h(\cdot, \cdot)$ is both coercive and continuous with corresponding constants  independent of $h$ and the position of $\Gamma$ in the mesh. Then, it is a textbook result (see, e.g., \cite{ern2013theory} or \cite[section~5]{guzman2016inf} for the case of $s_h\neq0$) that the finite element formulation  \eqref{discrete} is well-posed  in the product norm $(\|\cdot\|_A^2+\|\cdot\|_h^2)^{\frac12}$, \emph{provided the following  holds}: There exists $c_0>0$ independent of $h$ and the position of $\Gamma$ in the mesh
such that
\begin{equation}\label{LBB}
  c_0\|q\|_h\le \sup_{\bv\in\bU_h}\frac{b_T(\bv,q)}{\|\bv\|_A} +s_h(q,q)^{\frac12}\quad\forall~q\in Q_h.
\end{equation}
We call this the \emph{inf-sup stability condition}.
Proving that this inf-sup stability condition is satisfied for trace  Taylor--Hood elements is the main topic of the paper and the subject of this section.

\begin{remark}[Cond. \eqref{LBB} is FE counterpart of \eqref{infsup}]\rm
Let us take a closer look at condition \eqref{LBB}, which we need for the well-posedness of the trace FEM \eqref{discrete}. For the norm on the left-hand side the inequality $\|q\|_h\ge \|q\|_{L^2(\Gamma)}$ trivially holds. Thanks to the Korn inequality~\eqref{korn1} and \eqref{equivA}, for the norm in the denominator we have the estimate~$c \|\bv\|_A \ge \|\bv_T\|_1$ for all $\bv \in \bU_h$. Therefore, \eqref{LBB} yields
\begin{equation*}
  \hat c_0\|q\|_{L^2(\Gamma)}\le \sup_{\bv\in\bU_h}\frac{b_T(\bv,q)}{\|\bv_T\|_1} +s_h(q,q)^{\frac12}\quad\forall~q\in Q_h,
\end{equation*}
with $\hat c_0>0$. The latter bound resembles  \eqref{infsup} for finite element spaces up to the term $s_h(q,q)^{\frac12}$, which depends on the \emph{normal} derivative of $q$ over the tetrahedra cut by $\Gamma$.
\end{remark}

\begin{remark}[$s_h(\cdot,\cdot)$ vs. common ``pressure-stabilization'']\rm \label{rem:stab} In the finite elements analysis of  the standard planar Stokes problem, it is common to add pressure stabilization in mixed finite element methods that do not satisfy the LBB condition, such as equal-order elements; see, e.g., \cite{john2016finite}. Such a stabilization also results in an additional bilinear $(p_h,q_h)$-form in the finite element formulation.
There is, however, an essential difference between such standard stabilizations of equal-order (or other LBB-unstable) finite element pairs and the volumetric normal pressure stabilization added in \eqref{discrete}.
For manifolds, such a standard pressure stabilization would mean the penalization of the \emph{tangential} variation of $p_h$, while   $s_h(p_h,q_h)$  defined in \eqref{forms} imposes a constraint on the \emph{normal} behaviour of $p_h$. For example, for the surface case the classical  Brezzi--Pitk\"{a}ranta stabilization~\cite{BP} is given by  $s_h^{\rm tang}(p,q)=\rho_p  \int_{\Gamma} \nablaG p \cdot \nablaG q \diff{s}$, with $\rho_p=O(h^2)$, or in the
form of a volumetric integral by $s_h^{\rm tang}(p,q)=\rho_p  \int_{\Omega_h^{\Gamma}} \nablaG p \cdot \nablaG q \diff{\vect x}$, with $\rho_p$ as in \eqref{parameters}. Combined with the normal volume stabilization $s_h(\cdot,\cdot)$, cf.~\eqref{forms}, one obtains a \emph{full pressure gradient} stabilization of the form
\begin{equation}\label{FullStab}
s_h^{\rm full}(p,q)=\rho_p  \int_{\Omega_h^{\Gamma}} \nabla p \cdot \nabla q \diff{\vect x}.
\end{equation}
This full pressure gradient stabilization has been used and analyzed in \cite{olshanskii2018finite} with $\vect P_1$--$P_1$ trace finite elements for the surface Stokes problem. Of course,  $s_h^{\rm full}(p,q)$ would also make the $\vect P_2$--$P_1$ trace FEM stable; however, due to a larger consistency error such a method does not have an optimal order discretization error. Numerical experiments (see section~\ref{sec:num}) show that our stability analysis presented below is sharp in the following sense: From the computed optimal constants $c_0$ in \eqref{LBB} we conclude that (i) for $\vect P_2$--$P_1$ trace FEM the discretization~\eqref{discrete} is  unstable for $s_h(p,q)=0$, but becomes stable with only the normal volume stabilization $s_h(p,q)$ as in \eqref{forms}; while (ii)  for  $\vect P_1$--$P_1$ trace FEM the discretization~\eqref{discrete} is  unstable for both $s_h(p,q)=0$ and  $s_h(p,q)$ as in \eqref{forms}, and  the full-gradient stabilization $s_h^{\rm full}(p,q)$ makes it stable.
\end{remark}
\medskip


We outline the structure of our analysis for proving the inf-sup stability condition $\eqref{LBB}$. In section~\ref{Sectequi} we present equivalent formulations of the inf-sup stability condition. One of these formulations  essentially follows from the so-called ``Verf\"urth's trick'', which is well-known in the  stability analysis of mixed finite element pairs \cite{verfurth1984error}. Based on this, another equivalent formulation is derived that uses the notion of \emph{regular elements}, which is known in the literature on trace FEM \cite{Alg1,DemlowOlsh}. The derivation of the latter equivalent formulation is based on a key new result (``neighborhood estimate'')  which essentially states that for finite element functions the $L^2$ norm on \emph{any} element $T \in \cT_h^\Gamma$ can be controlled by the $L^2$ norm on a neighboring \emph{regular} element and the $L^2$ norm of the normal derivative (i.e., normal to the surface) in a small neighborhood. This result may be useful also in other analyses of
trace finite element methods. A proof of this neighborhood estimate is given in a separate section~\ref{sectproof}. The results concerning the equivalent formulations of the inf-sup stability condition and the neighborhood estimate are valid for surface Taylor--Hood pairs  for all $k \geq 1$. The formulation of the inf-sup stability condition in terms of regular elements is tailor-made for our setting and in section~\ref{sectmainresult} we show that it is satisfied for $k=1$, i.e., for the $\vect P_2$--$P_1$ surface Taylor--Hood pair.

In the remainder of the paper we write $x\lesssim y$ to state that the inequality  $x\le c y$
holds for quantities $x,y$ with a constant $c$, which is independent of the mesh parameter
$h$ and the position of $\Gamma$ in the background mesh.
Similarly for $x\gtrsim y$, and $x\simeq y$ means that both $x\lesssim y$ and $x\gtrsim y$ hold.

\subsection{Equivalent formulations of the inf-sup stability condition}\label{Sectequi}
 The following lemma is an application of Verf\"urth's trick \cite{verfurth1984error} in the setting of trace finite element methods. We make use of  the  following local trace inequality, cf.~\cite{Hansbo02,reusken2015analysis,guzman2016inf}:
\begin{equation} \label{fund1B}
 h_T \|v\|_{L^2(\Gamma_T)}^2 \lesssim  \|v\|_{L^2(T)}^2+h_T^2\|v\|_{H^1(T)}^2 \quad \text{for all}~v \in H^1(T),~T\in\T_h^\Gamma,
\end{equation}
with $\Gamma_T \coloneqq \Gamma \cap T$.
We further need the following norm on $Q_h$:
\begin{equation}\label{norm1}
  \|q\|_{1,\,h}\coloneqq \Big(\sum_{T \in \T_h^\Gamma} h_T \|\nabla q\|_{L^2(T)}^2\Big)^{\frac12}.
\end{equation}

\begin{lemma}\label{lem:equiv}
The inf-sup stability condition \eqref{LBB} is equivalent to
\begin{align}
  \|q\|_{1,\,h} & \lesssim \sup_{\bv\in\bU_h}\frac{b_T(\bv,q)}{\|\bv\|_A} +s_h(q,q)^{\frac12}\quad\forall~q\in Q_h.\label{LBB2}
\end{align}
\end{lemma}
\begin{proof}
From a finite element inverse inequality and~\eqref{L2control} we get
 \[ \Big(\sum_{T \in \T_h^\Gamma} h_T \|\nabla q\|_{L^2(T)}^2\Big)^{\frac12} \lesssim h^{-\frac12} \|q\|_{L^2(\OGamma)} \lesssim \|q\|_{h} \quad \text{for all}~q \in Q_h.
\]
Hence, \eqref{LBB} implies \eqref{LBB2}.

We now derive \eqref{LBB2} $\Rightarrow$ \eqref{LBB}.
Take $ q \in Q_h$. Thanks to the inf-sup property \eqref{infsup}, there exists
$\bv\in\bV_T$ such that
\begin{equation} \label{aux10}
  b_T( \bv,  q)= \| q\|_{L^2(\Gamma)}^2,\qquad \|\bv\|_1\lesssim \| q\|_{L^2(\Gamma)}.
\end{equation}
We consider  $\bv^e\in H^1(\mathcal{O}(\Gamma))$, a normal  extension of $\bv$ off the surface to a neighborhood $\mathcal{O}(\Gamma)$ of width $O(h)$ such that  $\OGamma\subset\mathcal{O}(\Gamma)$.
For this normal extension one has (see, e.g., \cite{OlshReusken08})
\begin{equation}\label{normal}
\|\bv^e\|_{H^1(\OGamma)}\simeq h^{\frac12}\|\bv\|_{H^1(\Gamma)}.
 \end{equation}
Take $\bv_h\coloneqq  I_h ( \bv^e) \in \bU_h$, where $I_h : H^1(\mathcal{O}(\Gamma))^3\to \bU_h$ is the Cl\'{e}ment interpolation operator.
By standard arguments (see, e.g., \cite{reusken2015analysis}) based  on stability and approximation properties of $I_h ( \bv^e)$, one gets
\begin{equation*}
\begin{split}
 \|\bv_h\|_A^2&=\|I_h ( \bv^e)\|_A^2\\
~ & \lesssim \|I_h ( \bv^e)\|_1^2+h^{-2}\|I_h ( \bv^e)_N\|_{L^2(\Gamma)}^2+ h^{-1} \|\nabla (I_h ( \bv^e))\bn\|_{L^2(\OGamma)}^2\\
{\footnotesize \bv\cdot\bn=0, \eqref{fund1B}}~  &\lesssim \sum_{T\in \cT_h^\Gamma} h_T^{-1}\|I_h ( \bv^e)\|_{H^1(T)}^2+h^{-2}\|\big(I_h ( \bv^e)-\bv\big)\cdot \bn\|_{L^2(\Gamma)}^2\\
{\footnotesize \eqref{fund1B}}~  &\lesssim \sum_{T\in \cT_h^\Gamma} h_T^{-1}\|\bv^e\|_{H^1(\omega(T))}^2+   h^{-2}  \sum_{T\in \cT_h^\Gamma} h_T^{-1}\|I_h ( \bv^e)-\bv^e\|_{L^2(T)}^2\\ &\qquad+
 h^{-2} \sum_{T\in \cT_h^\Gamma} h_T \|I_h ( \bv^e)-\bv^e\|_{H^1(T)}^2\\
 &\lesssim    \sum_{T\in \cT_h^\Gamma} h_T^{-1}\|\bv^e\|_{H^1(\omega(T))}^2\lesssim    h^{-1}\|\bv^e\|_{H^1(\OGamma)}^2\\
{\footnotesize \eqref{normal}}~  & \lesssim \|\bv\|_1^2.
 \end{split}
\end{equation*}
Hence due to \eqref{aux10} we obtain
\begin{equation} \label{aux1U}
 \|\bv_h\|_A\lesssim \|q\|_{L^2(\Gamma)}.
\end{equation}
Using \eqref{fund1B} and approximation properties of $I_h ( \bv^e)$  one gets
\begin{equation}\label{aux1d}
 \|\bv - I_h ( \bv^e)\|_{L^2(\Gamma)}\lesssim h\|\bv\|_{H^1(\Gamma)}.
\end{equation}
Using \eqref{aux1d} and \eqref{aux10}, \eqref{fund1B} we obtain
\[
\begin{split}
b_T(\bv_h,  q)  & = b_T( \bv,  q)- b_T(\bv - I_h ( \bv^e), q) \\
&\ge  \| q\|_{L^2(\Gamma)}^2 - \|\bv - I_h ( \bv^e)\|_{L^2(\Gamma)} \|\gradG q\|_{L^2(\Gamma)}\\
{\footnotesize \eqref{aux1d}}~
&\ge  \| q\|_{L^2(\Gamma)}^2 - c h \|\bv\|_{H^1(\Gamma)} \|\gradG q\|_{L^2(\Gamma)} \\ & = \| q\|_{L^2(\Gamma)}^2 - c h \|\bv\|_{H^1(\Gamma)} \left(\sum_{T\in \cT_h^\Gamma}\|\gradG q \|_{L^2(\Gamma_T)}^2\right)^{\frac12}\\
{\footnotesize \eqref{fund1B}}~
& \ge  \| q\|_{L^2(\Gamma)}^2 - c \,h^{\frac12} \|\bv\|_{H^1(\Gamma)} \left(\sum_{T\in \cT_h^\Gamma}\|\nabla q \|_{L^2(T)}^2\right)^{\frac12} \\
{\footnotesize \eqref{aux10}}~
 & \ge  \| q\|_{L^2(\Gamma)}^2 - c \|q\|_{L^2(\Gamma)} \|q\|_{1,\,h}. 
\end{split}
\]
This and \eqref{aux1U} yield
\begin{equation}\label{aux552}
 \| q\|_{L^2(\Gamma)} - c  \|q\|_{1,\,h} \lesssim \sup_{\bv\in\bU_h}\frac{b_T(\bv,q)}{\|\bv\|_A}.
\end{equation}
From \eqref{LBB2} and \eqref{aux552} we have
\begin{align*}
  \|q\|_{L^2(\Gamma)} & \lesssim \sup_{\bv\in\bU_h}\frac{b_T(\bv,q)}{\|\bv\|_A} +s_h(q,q)^{\frac12}\quad\forall~q\in Q_h, 
  \end{align*}
which implies \eqref{LBB}.
\end{proof}

We now derive a further condition that is equivalent to \eqref{LBB2}, in which the norm on the left-hand side in \eqref{LBB2} is replaced by a weaker one in which $\sum_{T \in \cT_h^\Gamma}$ is replaced by $\sum_{T \in \cT^{\Gamma}_{\textnormal{reg}}}$ with
$\cT^{\Gamma}_{\textnormal{reg}} \subset  \cT_h^\Gamma$ a subset of ``regular elements.''
The following notion of regular elements appeared earlier in the literature on trace FEM~\cite{DemlowOlsh,Alg1}.
We define the set of \emph{regular elements} as  those $T \in \T_h^\Gamma$ for which the area of the intersection $\Gamma_T = \Gamma \cap T$ is not less than
$\hat c_\cT h_T^2$ with some sufficiently small   threshold  parameter $ \hat c_\cT > 0$, whose  value  will be specified later:
\begin{equation} \label{defregGglobal}
  \cT^{\Gamma}_{\textnormal{reg}} \coloneqq \{\, T \in \cT_h^\Gamma\::\:|\Gamma_T| \geq \hat c_\cT h_T^2\,\}.
\end{equation}

The set $\T^\Gamma_{\textnormal{reg}}$ is ``dense'' in $\T_h^\Gamma$ the following sense: Every $T\in\T_h^\Gamma$ has a regular element in the set of its neighboring tetrahedra $\omega(T)\coloneqq
\{\,T' \in \cT_h^\Gamma\::\:\overline{T'} \cap \overline{T} \neq \emptyset\,\}$, cf. section~\ref{sectproof}.  Using this property the result in the following key lemma can be proved, which shows that for any $T\in\T_h^\Gamma$ and $q \in Q_h$ the norm $\|q\|_{L^2(T)}$  can be essentially controlled by $\|q\|_{L^2(T')}$ for a  neighbouring regular element $T'$ and normal  derivatives in a small volume neighborhood.
In addition to the norm defined in \eqref{norm1}, it is convenient to introduce the following \emph{semi}norm on $Q_h$:
\[
\|q\|_{1,\,\textnormal{reg}}\coloneqq \Big(\sum_{T \in \T_{\textnormal{reg}}^\Gamma} h_T \|\nabla q\|_{L^2(T)}^2\Big)^{\frac12}.
\]
\begin{lemma}[Neighborhood estimate] \label{Lemma_main}
 There exists a constant $c$ (depending only on shape regularity properties of $\cT_h$ and the (local) smoothness of $\Gamma$) such that for each $T \in \cT_h^\Gamma$ there exists $T' \in  \omega(T)\cap \cT^{\Gamma}_{\textnormal{reg}}$ and
 \begin{equation} \label{mainest}
  \|q\|_{L^2(T)} \leq c \Big( \|q\|_{L^2(T')} + h_T\|\vect n \cdot \nabla q\|_{L^2(\omega(T))} + h_T^2 \|\nabla q\|_{L^2(\omega(T))}\Big) \quad \forall~q \in Q_h.
 \end{equation}
\end{lemma}

A proof of this result is given in section~\ref{sectproof}.
\begin{remark}[On estimate \eqref{mainest}]\rm To see the improvement offered by \eqref{mainest} over available results, it is instructive to compare \eqref{mainest} to a local Sobolev inequality, which is proved by a different argument (cf. Lemma~3.1 in~\cite{OlshReusken08}):
\begin{equation} \label{E4}
\|v\|_{L^2(T)}\le \|v\|_{L^2(\omega(T))} \leq c \Big( \|v\|_{L^2(T')} + h_T\|\nabla v\|_{L^2(\omega(T))} \Big) \quad \forall~v \in H^1(\omega(T)).
\end{equation}
The latter result holds for $v \in H^1(\omega(T))$, while in \eqref{mainest} we restrict to $q\in Q_h$. In \eqref{E4} the first order (in $h_T$) term contains the (full) gradient, whereas in \eqref{mainest} only the normal derivative is needed.
\end{remark}
\ \\[1ex]
The following corollary is needed in the proof of Lemma~\ref{lemequi2} below.
\begin{corollary}\label{corol}
 For any $T \in \cT_h^\Gamma$ there exists $T' \in \omega(T)\cap\cT^{\Gamma}_{\textnormal{reg}}$ such that
 \begin{equation} \label{mainest2}
  \|\nabla q\|_{L^2(T)} \lesssim \| \nabla q\|_{L^2(T')} + \|\vect n \cdot \nabla q\|_{L^2(\omega(T))} + h_T \|\nabla q\|_{L^2(\omega(T))} \quad \forall~q \in Q_h.
 \end{equation}
 Furthermore, for $h$ sufficiently small we have
 \begin{equation} \label{mainest3}
  \|q\|_{1,\,h}^2 \lesssim \|q\|_{1,\,\textnormal{reg}}^2 +s_h(q,q) \quad \forall~q \in Q_h.
 \end{equation}
\end{corollary}
\begin{proof}
 Take $T \in \cT_h^\Gamma$  and the corresponding  $T'  \in  \omega(T)\cap \cT^{\Gamma}_{\textnormal{reg}}$ as in Lemma~\ref{Lemma_main}.  Take $q \in Q_h$ and define $c_0 \coloneqq \frac{1}{|T'|}\int_{T'} q \diff{\vect x}$.
Note the (local) Poincare inequality $\|q-c_0\|_{L^2(T')}\lesssim h_T \|\nabla q\|_{L^2(T')}$. Using this, a finite element inverse inequality, and \eqref{mainest} we obtain
\begin{align*}
 \|\nabla q\|_{L^2(T)}& = \|\nabla (q-c_0)\|_{L^2(T)} \lesssim h_T^{-1} \|q-c_0\|_{L^2(T)} \\
  & \lesssim  h_T^{-1}\|q-c_0\|_{L^2(T')} + \|\vect n \cdot \nabla q\|_{L^2(\omega(T))} + h_T \|\nabla q\|_{L^2(\omega(T))}\\
  &  \lesssim  \|\nabla q \|_{L^2(T')} + \|\vect n \cdot \nabla q\|_{L^2(\omega(T))} + h_T \|\nabla q\|_{L^2(\omega(T))},
\end{align*}
which is the desired estimate~\eqref{mainest2}. Squaring and multiplying \eqref{mainest2} by $h_T$, summing over $T\in \cT_h^\Gamma$ and using a finite overlap property we obtain
\[
  \|q\|_{1,\,h}^2 \lesssim \|q\|_{1,\,\textnormal{reg}}^2 + s_h(q,q) + h^2 \|q\|_{1,\,h}^2.
\]
For $h$ sufficiently small this yields the result \eqref{mainest3}.
\end{proof}
\ \\[1ex]
These results are used to derive a condition that is equivalent to \eqref{LBB2}.
\begin{lemma} \label{lemequi2} For $h$ sufficiently small, the condition \eqref{LBB2} is equivalent to the following one:
 \begin{align}
  \|q\|_{1,\,\textnormal{reg}} & \lesssim \sup_{\bv\in\bU_h}\frac{b_T(\bv,q)}{\|\bv\|_A} +s_h(q,q)^{\frac12}\quad\forall~q\in Q_h.\label{LBB3}
\end{align}
\end{lemma}
\begin{proof}
Clearly, \eqref{LBB2} implies \eqref{LBB3}.
The reverse direction follows from \eqref{mainest3}.
\end{proof}

\subsection{Proof of Lemma~\ref{Lemma_main}} \label{sectproof}

For proving that for each $T\in \cT_h^\Gamma$ there exists $T' \in \omega(T)\cap \cT_{\textnormal{reg}}^\Gamma$ such that \eqref{mainest} holds we use a construction based on a local  graph representation of $\Gamma$ over a tangent plane.

Consider an arbitrary $ T \in \cT_h^\Gamma$. Due to shape regularity, the number of elements in $\omega(T)$ is uniformly bounded by some constant~$K_\cT$.  Furthermore, there exists a constant $c_{1,\cT} \in (0,1]$ that depends
only on the shape regularity property such that
\begin{equation}\label{c1}
	B(\vect x; c_{1,\cT} h_T ) \cap \Omega_h^\Gamma \subset \omega(T)\quad  \forall ~\vect x \in T.
\end{equation}
Here and further $B(\vect x; r) \subset \mathbb{R}^3$ is the ball of radius $r$ centered at $\vect x$. Let $L$ be a given plane. The orthogonal projection of $T$ on $L$ is denoted by
$P_L(T)$. This projection is either a triangle or a convex quadrilateral. From elementary geometry it follows that all interior angles of $P_L(T)$ are bounded away from zero and the lower bound, which is independent of $L$ and $T$, depends only on shape regularity properties of $\cT_h$. This implies that there exists a strictly positive constant $c_{2,\cT}$, independent of $L$ and $T$, but dependent on shape regularity properties, such that
\begin{equation} \label{c2}
	\frac{|B(\vect x;r h_T) \cap P_L(T)|}{r^2 h_T^2} \geq c_{2,\cT} \quad \text{for all}~\vect x \in P_L(T)~,~r \in (0,1].
\end{equation}
Take $\vect x_0 \in \Gamma_T=T\cap \Gamma$. Let $\lPlane$ be the tangential plane at $\vect x_0$. We define a local (in $\omega(T)$) Euclidean coordinate system with origin at $\vect x_0$ and such that $\vect x = (z_1, z_2, z_3) \in \lPlane$ iff $z_3 = 0$. We write $\vect z \coloneqq (z_1, z_2, 0) \in \lPlane$. For $h$ sufficiently small, the surface $\Gamma\cap \omega(T)$ can be represented as a graph over $\lPlane$. We denote the graph function by $g$, i.e $\vect x \in \Gamma \cap \omega(T)$ can be represented as $\vect x = \vect z + g(\vect z) \vect n_{\vect x_0}$ with $\vect z \in \lPlane$ and $\vect n_{\vect x_0}$ a unit normal on $\lPlane$. We  have $g(\vect 0) = 0$ and
\begin{equation} \label{estg}
	\|g\|_{L^\infty(\omega(T) \cap \lPlane)}  \leq c_\Gamma h_T^2,
\end{equation}
with a constant $c_\Gamma$ that depends only on the local smoothness of $\Gamma$. We assume that $h_T$ is sufficiently small such that
\begin{equation} \label{esth}
	c_\Gamma h_T \leq \frac12 c_{1,\cT}
\end{equation}
holds with $c_{1,\cT}$ from \eqref{c1}. A projection $P_\Gamma(T) \subset \Gamma$ is defined in two steps. First, $P_{\lPlane}(T)$ is the orthogonal projection of $T$ on the plane $\lPlane$ as introduced above. Then the projection $P_\Gamma(T)$ is defined as the graph of $\Gamma$ over the projection $P_{\lPlane}(T)$ (cf. Fig.~\ref{Fig:illust}):
\[
	P_\Gamma(T) \coloneqq \{\,\vect z + g(\vect z) \vect n_{\vect x_0}\::\:\vect z \in P_{\lPlane}(T)\,\}.
\]

The set of \emph{local regular elements}  is defined as follows:
\begin{equation} \label{defregG}
	\omega^{\Gamma}_{\textnormal{reg}}(T)\coloneqq  \{\, T' \in \omega(T)\::\:|T' \cap P_{\Gamma}(T)| \geq \hat c_\cT h_T^2\,\}, \quad \hat c_\cT\coloneqq  \frac14 K_\cT^{-1} c_{1,\cT}^2 c_{2,\cT}.
\end{equation}

\begin{lemma} \label{lemneighG}
	For $T \in \cT_h^\Gamma$ the set $\omega^\Gamma_{\textnormal{reg}}(T)$ is nonempty.
\end{lemma}
\begin{proof}
Let $\vect x_0 \in \Gamma_T$ be as above. Note that $\vect x_0 = \vect 0$ in the local coordinate system. Define
\[
	B_{\lPlane}^\ast \coloneqq B(\vect x_0; \frac12 c_{1,\cT} h_T) \cap P_{\lPlane}(T).
\]
Note that due to \eqref{c2} we have
\begin{equation} \label{BL2}
	|B^\ast_{\lPlane}| \geq \frac14c_{2,\cT} c_{1,\cT}^2 h_T^2.
\end{equation}
We lift $B_{\lPlane}^\ast$ to the surface $\Gamma$
\[
	B_\Gamma^\ast \coloneqq \{\,\vect z + g(\vect z) \vect n_{\vect x_0}\::\:z \in B_{\lPlane}^\ast\,\} \subset P_\Gamma(T).
\]
Using~\eqref{estg} and \eqref{esth}, for $\vect x = \vect z + g(\vect z) \vect n_{\vect x_0} \in B_\Gamma^\ast$ we get
\[
	\|\vect x - \vect x_0\| = \|\vect x\| \leq \|\vect z\| + |g(\vect z)| \leq \frac12 c_{1,\cT} h_T + c_\Gamma h_T^2 \leq c_{1,\cT} h_T.
\]
Hence, $B_\Gamma^\ast \subset \omega(T)$. Noting that $|B_\Gamma^\ast| \geq |B_{\lPlane}^\ast|$ and using \eqref{BL2}, we get
\[
	|B_\Gamma^\ast| \geq \frac14 c_{2,\cT} c_{1,\cT}^2 h_T^2.
\]
Due to $B_\Gamma^\ast \subset \omega(T)$ we have
$
	B_\Gamma^\ast = \bigcup\limits_{T' \in\omega(T)} (B_\Gamma^\ast \cap T')
$
and
\[
	\frac14 c_{2,\cT} c_{1,\cT}^2 h_T^2 \leq  |B_\Gamma^\ast| \leq K_\cT \max_{T' \in  \omega(T)} |B_\Gamma^\ast \cap T'| \leq K_\cT \max_{T' \in  \omega(T)} |P_\Gamma(T) \cap T'|.
\]
This implies that there exists $T' \in \omega(T)$ with $|P_\Gamma(T) \cap T'| \geq  \hat c_\cT h_T^2$, with $\hat c_\cT=\frac14 K_\cT^{-1} c_{1,\cT}^2 c_{2,\cT}$.
\end{proof}

\subsubsection{Completing the proof of Lemma~\ref{Lemma_main}}
We shall need the following simple technical result.
\begin{lemma}\label{Lem:AreaPer} Let $S\subset\mathbb{R}^2$ be a simply-connected domain, such that $|S|=1$ and $\partial S$ is a rectifiable curve of length $L$. Then there is a disc $D\subset S$, such that  $\mbox{radius}(D)\ge c>0$, where $c$ depends only on $L$.
\end{lemma}
\begin{proof} For some fixed $N \in \Bbb{N}$, consider the regular tiling  $\mathcal{G}$ of $\mathbb{R}^2$ with squares of size  $\frac1N\times\frac1N$. Consider the subsets $\mathcal{G}_S=\{K\in\mathcal{G}: |K\cap S|>0\}$ and
$\mathcal{G}_{\partial S}=\{K\in\mathcal{G}: {K}\cap \partial S\neq\emptyset\}$.
Since $|S|=1$ and $S\subset\bigcup_{K\in\mathcal{G}_S}\overline{K}$ we have  $\#(\mathcal{G}_S)\ge N^2$. From the simple observation that a piece of curve of length $\frac1N$ intersects at most 4 squares one gets the bound
$\#(\mathcal{G}_{\partial S})\le 4N\lceil L\rceil$; see, e.g.,~\cite{hunter1979operations}.
Therefore, taking $N=4\lceil L\rceil+1$ we have at least one $K\in\mathcal{G}_S\setminus\mathcal{G}_{\partial S}$ and so there is $D\subset S$ with $\mbox{radius}(D)\ge \frac12(4\lceil L\rceil+1)^{-1}$.
\end{proof}
\ \\[1ex]
Consider an arbitrary $T\in \cT^{\Gamma}_h$. The constant $\hat c_\cT$, defined in \eqref{defregG}, is the one that we use in the definition of the set $\T_{\textnormal{reg}}^\Gamma$ in \eqref{defregGglobal}.
For $T' \in  \omega^{\Gamma}_{\textnormal{reg}}(T)$ we have
\begin{equation} \label{largeest}
	|\Gamma_{T'}|= |T'\cap \Gamma| \geq |T' \cap P_\Gamma(T)| \geq \hat c_\cT h_T^2.
\end{equation}
Therefore, we have
\begin{equation} \label{inclusion}
 \omega^\Gamma_{\textnormal{reg}}(T) \subset \cT^{\Gamma}_{\textnormal{reg}} \quad \text{for all}~~T \in \cT_h^\Gamma.
\end{equation}

Take $q \in Q_h$. The surface $P_\Gamma(T)$ is the graph of a function $g$ on $P_{\lPlane}(T)$. Hence, there is a subset $S \subset P_{\lPlane}(T)$ such that
\[
	T' \cap P_\Gamma(T) = \{\,\vect z + g(\vect z) \vect n_{\vect x_0}\::\:\vect z \in S\,\}.
\]
 Using the surface area formula for the graph and  \eqref{largeest}, we get
\[
	\hat c_\cT h_T^2\le |T' \cap P_\Gamma(T)| = \int_S\sqrt{1+\|\nabla g\|^2}\diff{s}\le \sqrt{1+\max_{S}\|\nabla g\|^2}|S|\le (1+ch^2)|S|,
\]
\begin{wrapfigure}{l}{.45\linewidth}
	\begin{subfigure}{1.\linewidth}
		\def\svgwidth{\linewidth}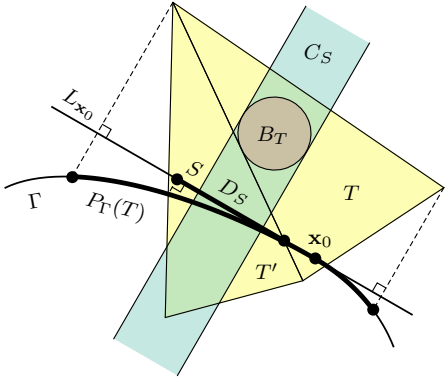
	\end{subfigure}%
	\caption{2D illustration for the proof of Lemma~\ref{Lemma_main}} \label{Fig:illust}
	\vskip .5cm
\end{wrapfigure}
where we used $\nabla g(\vect x_0) = \vect 0$, $\vect x_0 \in \mathcal{O}_h(S)$, and $\|\nabla g\|_{L^\infty(S)} \lesssim h$ due to the $C^2$-smoothness of $\Gamma$. Hence, we obtain
$
	|S| \gtrsim h_T^2.
$
Using assumption~\eqref{ass1}, we estimate the perimeter of $S$:
\[
|\partial S|\lesssim |\partial (T' \cap P_\Gamma(T))| \lesssim h_T.
\]
From assumption~\eqref{ass1} we  have that $T' \cap \Gamma$, hence also $T' \cap P_\Gamma(T) $,  is simply-connected,  which implies that $S$ is  simply-connected. We now combine the result of Lemma~\ref{Lem:AreaPer} with a scaling argument to find a disc $D_S\subset S$ such that
\begin{equation}\label{D_S}
	\mbox{radius}(D_S) \gtrsim h_T.
\end{equation}
For the \textit{circular} cross-section $D_S$ we define a corresponding cylinder:
\[
	C_S \coloneqq \{\,\vect z + \alpha \vect n_{\vect x_0}\::\:\vect z \in D_S, ~\alpha \in \Bbb{R}\,\}. 
\]
Due to  $D_S \subset P_{\lPlane}(T)$,  \eqref{D_S} and the shape regularity of $T$  one can inscribe a ball $B_T\subset C_S \cap T$ such that  $\mbox{radius}(B_T) \gtrsim h_T$;
 see Figure~\ref{Fig:illust}.
By a standard scaling and norm equivalence argument\footnote{We outline the argument: By  norm equivalence in a finite dimensional space one obtains $\|q\|_{B(R)}\le C \|q\|_{B(1)}$, where $B(R)$ is ball of radius $R$, $q\in \mathbb{P}_n(\mathbb{R}^3)$, and $C$ depends only on $R$ and $n$. Further one uses a linear scaling to map $B(1)$ to $B_T$ and sets $R$  large enough, but independent of $T$ and $h$, such that $T$ always lies in the image of $B(R)$.} it follows that
\[
	\|q\|_{L^2(T)} \lesssim \|q\|_{L^2(B_T)}
\]
holds.
 Using this we obtain
\begin{align*}
	\|q\|_{L^2(T)}^2 & \lesssim \|q\|_{L^2(B_T)}^2 \lesssim \|q\|_{L^2(C_S \cap T)}^2 \lesssim \|q\|_{L^2(C_S\cap \omega(T))}^2 \\ & \lesssim h_T \|q\|_{L^2(T' \cap P_\Gamma(T))}^2 + h_T^2\|\vect n_{\vect x_0}\cdot \nabla q\|_{L^2(C_S\cap \omega(T))}^2 \\
 	& \lesssim h_T \|q\|_{L^2(\Gamma_{T'})}^2 + h_T^2\|\vect n_{\vect x_0}\cdot \nabla q\|_{L^2(\omega(T))}^2.
\end{align*}

Combining this with $h_T\|q\|_{L^2(\Gamma_{T'})}\lesssim \|q\|_{L^2(T')}$ (which follows from \eqref{fund1B} and a FE inverse inequality) and  $\|\vect n_{\vect x_0} - \vect n(\vect y)\| \lesssim h_T$ for all $\vect y \in \omega(T)$, we have
\[
	\|q\|_{L^2(T)}^2 \lesssim \|q\|_{L^2(T')}^2 + h_T^2 \|\vect n \cdot \nabla q\|_{L^2(\omega(T))}^2 + h_T^4 \|\nabla q\|_{L^2(\omega(T))}^2,
\]
which completes the proof.

\subsection{Discrete inf-sup condition is satisfied for trace $\vect P_2$--$P_1$ finite elements}\label{sectmainresult}

We are now ready to derive the main result of our stability analysis. We  will show that the inf-sup stability condition \eqref{LBB3} holds for the case of trace $\vect P_2$--$P_1$ Taylor--Hood finite elements. In the proof we construct a velocity function from $\bU_h$, which delivers control over pressure gradients for all  \emph{regular} tetrahedra. Let $h_0>0$ be sufficiently small depending of $\Gamma$, but not on how $\Gamma$ intersects the background mesh.

\begin{theorem} \label{lemma2}
Consider $k=1$ in \eqref{TaylorHood}, i.e., $\bU_h$, $Q_h$ is the $\vect P_2$--$P_1$ Taylor--Hood pair. Then the inf-sup stability condition \eqref{LBB} is satisfied for $h\le h_0$.
\end{theorem}

\begin{proof} Below we prove \eqref{LBB3}. Due to the results in the Lemmas~\ref{lem:equiv} and~\ref{lemequi2} this implies \eqref{LBB}.
Denote by  $\mathcal{E}_{\textnormal{reg}}$ the set of all edges of tetrahedra from  $\T^\Gamma_{\textnormal{reg}}$.
Let $\widetilde{\mathbf{t}}_E$ be a vector connecting the two endpoints of $E \in \mathcal{E}_{\textnormal{reg}}$ and $\mathbf{t}_E \coloneqq \widetilde{\mathbf{t}}_E/|\widetilde{\mathbf{t}}_E|$.  For each edge $E$ let  $\phi_E$ be the  quadratic nodal finite element basis function corresponding to the midpoint of $E$.
For  $q\in Q_h$ define  $\bv \in \bU_h$ as follows:
\begin{equation}\label{VforQregA}
	\bv(\vect x) \coloneqq \sum_{E\in\mathcal{E}_{\textnormal{reg}}} h_E^2\phi_E(\vect x)\, [\mathbf{t}_E\cdot\nabla q(\vect x)]\mathbf{t}_E.
\end{equation}
using $0 \leq \phi_E \leq 1$ in $T \in \T_h^{\Gamma}$, we obtain
\begin{align}
	& (\bv,\nabla_\Gamma q)_{L^2(\Gamma_T)}=(\bv,\bP\nabla q)_{L^2(\Gamma_T)} \nonumber\\
	&= \int_{\Gamma_T}\sum_{E\in\mathcal{E}_{\textnormal{reg}}}h_E^2\phi_E\, |\bP\mathbf{t}_E\cdot\nabla q|^2 \diff{s}+\int_{\Gamma_T}\sum_{E\in\mathcal{E}_{\textnormal{reg}}}h_E^2\phi_E\, (\bP^\perp\mathbf{t}_E\cdot\nabla q) (\bP\mathbf{t}_E\cdot\nabla q) \diff{s} \nonumber \\
	&\ge\frac12\int_{\Gamma_T} \sum_{E\in\mathcal{E}_{\textnormal{reg}}}h_E^2\phi_E\, |\bP\mathbf{t}_E\cdot\nabla q|^2 \diff{s} -\frac12\int_{\Gamma_T}\sum_{E\in\mathcal{E}_{\textnormal{reg}}}h_E^2\phi_E\, |\bP^\perp\mathbf{t}_E\cdot\nabla q|^2 \diff{s}\nonumber \\
	&\ge\frac12\int_{\Gamma_T} \sum_{E\in\mathcal{E}_{\textnormal{reg}}}h_E^2\phi_E\, |\bP\mathbf{t}_E\cdot\nabla q|^2\diff{s} -\frac12\int_{\Gamma_T}\sum_{E\in\mathcal{E}(T)}h_E^2 |\bn\cdot\nabla q|^2\diff{s} \nonumber \\
	&\ge\frac12\int_{\Gamma_T} \sum_{E\in\mathcal{E}_{\textnormal{reg}}}h_E^2\phi_E\, |\bP\mathbf{t}_E\cdot\nabla q|^2\diff{s} -3h_T^2 \|\bn\cdot\nabla q\|^2_{L^2(\Gamma_T)} \nonumber \\
	& \geq \frac12\int_{\Gamma_T} \sum_{E\in\mathcal{E}_{\textnormal{reg}}}h_E^2\phi_E\, |\bP\mathbf{t}_E\cdot\nabla q|^2\diff{s}-c_1(h_T\|\bn\cdot\nabla q\|_{L^2(T)}^2
	+ h_T^3\|\nabla q\|_{L^2(T)}^2). \label{est1}
\end{align}
For the last bound we used the local trace inequality \eqref{fund1B} and $\nabla^2q = \vect 0$ for the piecewise  linear function $q$.
Hence, for every $T \in \T_h^{\Gamma}$ we have
\begin{equation} \label{est2}
	(\bv,\nabla_\Gamma q)_{L^2(\Gamma_T)}+ c_1 h_T\|\bn\cdot\nabla q\|_{L^2(T)}^2
	\gtrsim  -h_T^3\|\nabla q\|_{L^2(T)}^2.
\end{equation}
We now restrict to $T \in \cT_{\textnormal{reg}}^\Gamma$ and estimate the first term in \eqref{est1}. Take  $T \in \cT_{\textnormal{reg}}^\Gamma$, i.e. $|\Gamma_T| \geq \hat c_\cT h_T^2$ holds. Let $\hat T$ be the unit tetrahedron and $G(\hat{\vect x}) = \vect A\,\hat{\vect x} + \vect b$ an affine mapping such that $G(\hat T) = T$. Define $\hat \Gamma_{\hat T}\coloneqq G^{-1}(\Gamma_T)$,  and for a function $\phi\,:\,T \to \Bbb{R}$, $\hat \phi \coloneqq \phi \circ G$. We then have
\begin{equation} \label{EE5}
	\int_{\Gamma_T} |\phi| \diff{s} \geq c_0 h_T^2 \int_{\hat \Gamma_{\hat T}} |\hat \phi | \diff{s}, \qquad |\hat \Gamma_{\hat T}| \geq \tilde c_\cT >0,
\end{equation}
with a constant $c_0 >0$ that depends only on shape regularity properties and $ \tilde c_\cT$ depending only on shape regularity properties and on $\hat c_\cT$ from \eqref{defregG}.

For the normal vector $\hat\bn$ on $\hat \Gamma_{\hat T}$ we have $\hat \bn =  \frac{\vect A^T\bn \circ G}{\|\vect A^{T}\bn \circ G\|}$. Using $\|\vect A\|=\|\nabla G\|\lesssim h$ and $\|\nabla\bn\|_{L^\infty(\Gamma_{T})} \lesssim 1$,  we check that $\|\nabla \hat \bn\|_{L^\infty(\hat \Gamma_{\hat T})} \lesssim h$. Normals on faces $\hat F$ of $\hat T$ are denoted by $\bn_{\hat F}$. For these normals we take the orientation the same as that of $\hat \bn$ on $\hat \Gamma_{\hat T}$. For $\hat \Gamma_{\hat T}$ we choose a corresponding \emph{base face} $\hat F_b$ as the one that  fits best to $\hat \Gamma_{\hat T}$ in the following sense:
\[
	\min_{ \hat \Gamma_{\hat T}} \|\hat \bn(\cdot)- \hat \bn_{\hat F_b}\| \leq \min_{ \hat \Gamma_{\hat T}} \|\hat \bn(\cdot)- \hat \bn_{\hat F}\|\text{ for all faces $\hat F$ of $\hat T$},
\]
cf. Figure~\ref{Fig:illust2} for a 2D illustration.
The face $F_b = G(\hat F_b)$ is called the base face of $T$. We take a fixed $\epsilon >0$ sufficiently small
such that $\|\nabla \hat \bn\|_{L^\infty(\hat \Gamma_{\hat T})} \leq \epsilon$ holds.
This means that $\hat \Gamma_{\hat T}$ is sufficiently close to a plane. This and the minimization property imply that for $\hat F\neq \hat F_b$ the angle between $\hat \bn$ and $\hat \bn_{\hat F}$  is uniformly bounded below from zero.
To see this, let $\vect x_0 \in  \hat \Gamma_{\hat T}$ be such that $\|\hat\bn(\vect x_0)-\hat\bn_{\hat F_b}\| = \min_{  \hat \Gamma_{\hat T}}\|\hat \bn(\vect x) - \hat \bn_{\hat F_b}\|$, then
\begin{equation}\label{angleBnd}
	\min_{\hat \Gamma_{\hat T}} \|\hat \bn(\cdot)- \hat \bn_{\hat F}\|\ge  \|\hat \bn(\vect x_0)- \hat \bn_{\hat F}\|
	- C \epsilon \ge \|\hat \bn_{\hat F_b} - \hat \bn_{\hat F}\| - \|\hat \bn(\vect x_0)- \hat \bn_{\hat F_b}\|
	- C \epsilon \gtrsim 1,
\end{equation}
for sufficiently small $\epsilon$.

For $\delta >0$ sufficiently small we define the reduced tetrahedron $\hat T_\delta \subset \hat T$
\begin{wrapfigure}{l}{.4\linewidth}
	\begin{subfigure}{1.\linewidth}
		\def\svgwidth{\linewidth}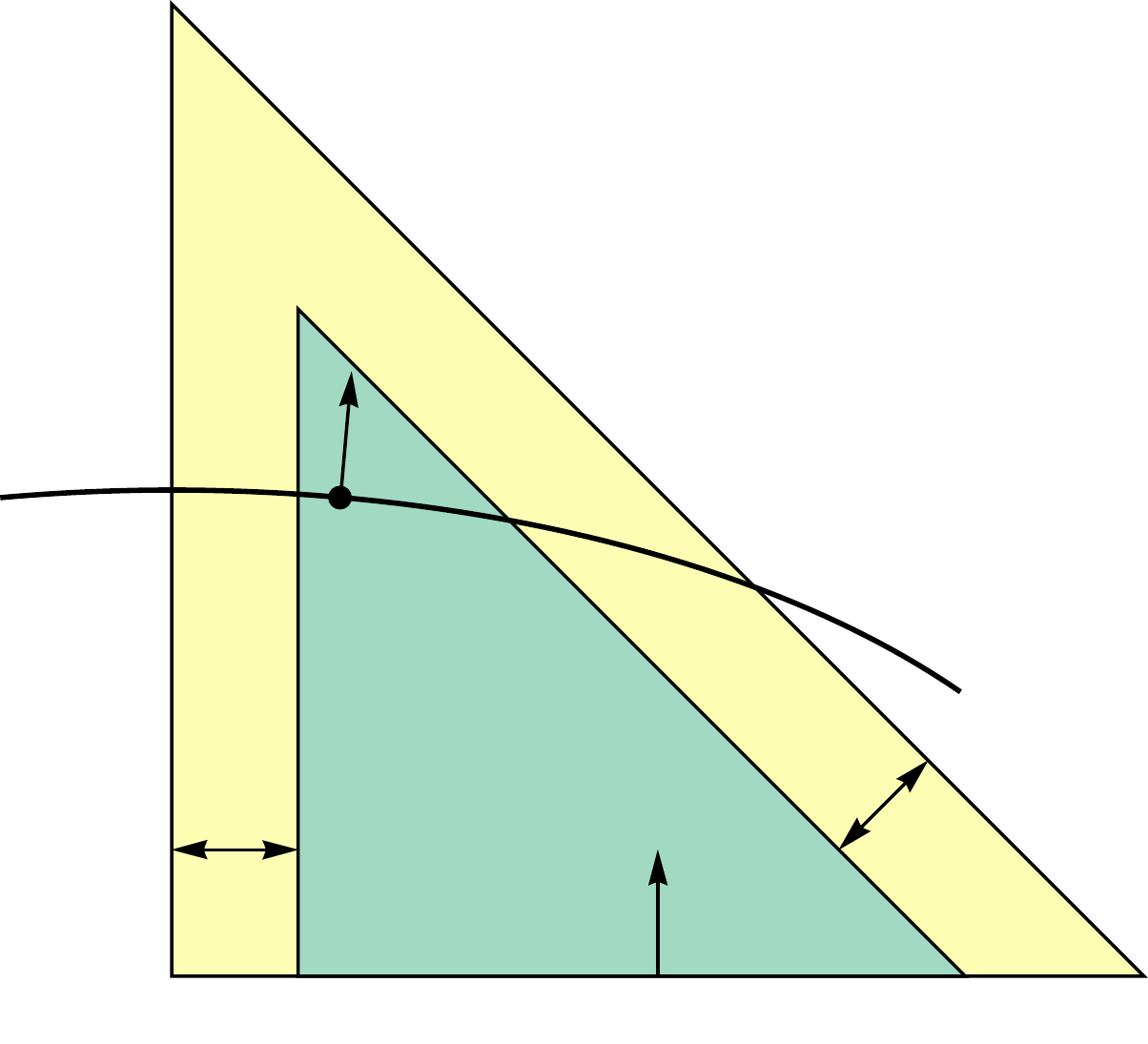
	\end{subfigure}%
	\caption{2D illustration of base face $\hat F_b$ and reduced tetrahedron~$\hat T_\delta$}
	\label{Fig:illust2}
\end{wrapfigure}
\[
	\hat T_\delta \coloneqq \{\,\vect x \in \hat T\::\:{\rm dist}(\vect x, \partial \hat T \setminus \hat F_b) \geq \delta\,\},
\]
cf. Figure~\ref{Fig:illust2}. Note that $\hat T_\delta$ depends on the base face $\hat F_b$. Thanks to \eqref{angleBnd} we can estimate
\[
	\left|\hat \Gamma_{\hat T}\cap \left(\hat T\setminus\hat T_{\delta_0}\right)\right|\lesssim \delta.
\]
Therefore, there exists $\delta_0 >0$, sufficiently small, such that
\[
  |\hat \Gamma_{\hat T} \cap \hat T_{\delta_0}| \geq \frac12 \tilde c_\cT
\]
holds. Let $E$ be an edge of the base face $F_b$ of $T$ and $\phi_E$ the corresponding nodal basis function defined above. We have $\hat \phi_E \geq c_1 \delta_0$ on $\hat T_{\delta_0}$, with a suitable generic constant $c_1>0$. Using this and \eqref{EE5} we obtain
\begin{equation} \label{e3} \begin{split}
	& \int_{\Gamma_T} \phi_E \diff{s}  = \int_{\Gamma_T} |\phi_E| \diff{s} \geq c_0  h_T^2 \int_{\hat \Gamma_{\hat T}} |\hat \phi_E| \diff{s} \\
	& \geq  c_0  h_T^2 \int_{\hat \Gamma_{\hat T} \cap \hat T_{\delta_0}} |\hat \phi_E| \diff{s}  \geq c_0  h_T^2 |\hat \Gamma_{\hat T} \cap \hat T_{\delta_0}| c_1\delta_0 \geq \frac12 c_0c_1 \delta_0 \tilde c_\cT h_T^2 =:C h_T^2,
\end{split}\end{equation}
where the constant $C >0$ is independent of $h$ and of how $\Gamma_T$ intersects $T$.

Consider $\vect x_0 \in \Gamma_T$ such that $\|\bn(\vect x_0)-\bn_{F_b}\| = \min_{\vect x \in \Gamma_T}\|\bn(\vect x)-\bn_{F_b}\|$ and the corresponding projector denoted by $\bP_0\coloneqq \bI- \bn(\vect x_0)\bn(\vect x_0)^T$. We have $\|\bP(\vect x)-\bP_0\| \lesssim h$ for $\vect x \in T$. Using this and \eqref{e3} we can estimate the first term in \eqref{est1} as follows:
\begin{equation} \label{E5_0}
\begin{split}
	\int_{\Gamma_T} \sum_{E\in\mathcal{E}_{\textnormal{reg}}}h_E^2\phi_E\,& |\bP\mathbf{t}_E\cdot\nabla q|^2\diff{s}
	 \gtrsim h_T^2 \sum_{E \subset F_b} \int_{\Gamma_T} \phi_E |\bP\mathbf{t}_E\cdot\nabla q|^2\diff{s} \\
	& \gtrsim h_T^2 \sum_{E \subset F_b}|\bP_0\mathbf{t}_E\cdot\nabla q_{|T}|^2 \int_{\Gamma_T}\phi_E \diff{s}  - c h_T^2 \|\nabla q\|_{L^2(T)}^2 \\
	& \gtrsim h_T^4 \sum_{E \subset F_b}|\bP_0\mathbf{t}_E\cdot\nabla q_{|T}|^2   - c h_T^2 \|\nabla q\|_{L^2(T)}^2.
\end{split}
\end{equation}
Due to the construction of the base face $F_b$ and the choice of $\vect x_0$, we have  that  $|\bn(\vect x_0) \cdot \bn_{F_b}| $ is uniformly bounded away from zero. This implies $\sum_{E \subset F_b}|\bP_0\mathbf{t}_E\cdot\nabla q_{|T}|^2 \gtrsim |\bP_0 \nabla q_{|T}|^2$. Using this we get
\begin{equation} \label{E5}
 \begin{split}
   \int_{\Gamma_T} \sum_{E\in\mathcal{E}_{\textnormal{reg}}}h_E^2\phi_E\, & |\bP\mathbf{t}_E\cdot\nabla q|^2\diff{s} \gtrsim h_T \|\bP_0 \nabla q\|_{L^2(T)}^2- c h_T^2 \|\nabla q\|_{L^2(T)}^2 \\
  &  \gtrsim h_T \|\bP \nabla q\|_{L^2(T)}^2- c h_T^2 \|\nabla q\|_{L^2(T)}^2 \\
  & \gtrsim h_T \|\nabla q\|_{L^2(T)}^2- h_T \|\bn\cdot \nabla q\|_{L^2(T)}^2 -c h_T^2 \|\nabla q\|_{L^2(T)}^2.
 \end{split}
\end{equation}
With the help of \eqref{E5} in \eqref{est1} we obtain for $T \in \cT_{\textnormal{reg}}^\Gamma$:
\begin{equation} \label{E6}
  (\bv,\nabla_\Gamma q)_{L^2(\Gamma_T)}+ c\, h_T\|\bn\cdot\nabla q\|_{L^2(T)}^2
\gtrsim   h_T \|\nabla q\|_{L^2(T)}^2 - c_2 h_T^2\|\nabla q\|_{L^2(T)}^2.
\end{equation}
Combining this with \eqref{est2} and summing over $T \in \cT_h$ yields
\[
 b_T(\bv,q) +s_h(q,q) \gtrsim \|q\|_{1,\,\textnormal{reg}}^2 - c h \|q\|_{1,\,h}^2,
\]
and combining this with \eqref{mainest3} we obtain (for $h$ sufficiently small)
\begin{equation}\label{E7}
 b_T(\bv,q) +s_h(q,q) \gtrsim \|q\|_{1,\,\textnormal{reg}}^2.
\end{equation}
We need the following elementary observation: For positive numbers $\alpha,\beta,\delta$ the inequality $\alpha + \beta^2 \geq c_0 \delta^2$ implies $\alpha + \beta(\beta +\delta) \geq \min\{c_0,1\} \delta(\beta +\delta)$ and thus $\frac{\alpha}{\beta +\delta} + \beta \geq \min\{c_0,1\} \delta$. Using this, the estimate \eqref{E7} implies
\begin{equation}\label{E8}
\frac{ b_T(\bv,q)}{ \|q\|_{1,\,\textnormal{reg}}+s_h(q,q)^\frac12} +s_h(q,q)^\frac12 \gtrsim \|q\|_{1,\,\textnormal{reg}}.
\end{equation}
It remains to estimate $\|\bv\|_A$.
We consider term by term the contributions in $\|\bv\|_A^2$.
Noting  $\|\nabla \phi_E\|_{\infty,T} \lesssim h_T^{-1}$ for $E \subset T$ and \eqref{fund1B} we get
 \begin{equation}\label{aux302}
  \begin{split}
  \|E_s(\bv)\|^2_{L^2(\Gamma)} &+\|\bv\|^2_{L^2(\Gamma)} \lesssim \sum_{T \in \cT_h^\Gamma} \|\nabla \bv\|_{L^2(\Gamma_T)}^2 + \|\bv\|_{L^2(\Gamma_T)}^2 \\
  &\lesssim  \sum_{T \in \cT_h^\Gamma}   h_T^{2}  \|\nabla q\|_{L^2(\Gamma_T)}^2
\lesssim  \sum_{T\in\T_h^\Gamma} h_T\|\nabla q\|_{L^2(T)}^2 = \|q\|_{1,\,h}^2.
 \end{split}
\end{equation}
We  also have for $\tau\lesssim h^{-2}$ and $\rho_u\lesssim h^{-1}$ the relations
 \begin{equation}\label{aux315}
 \begin{aligned}
  \tau\|v_N\|^2_{L^2(\Gamma)}&\leq \tau\|\bv\|^2_{L^2(\Gamma)} = \tau\sum_{T \in \cT_h^\Gamma} \|\bv\|_{L^2(\Gamma_T)}^2 \\
   & \lesssim \tau \sum_{T \in \cT_h^\Gamma} h_T^4 \|\nabla q\|_{L^2(\Gamma_T)}^2 \lesssim
   \tau \sum_{T \in \cT_h^\Gamma} h_T^3 \|\nabla q\|_{L^2(T)}^2 \lesssim \|q\|_{1,\,h}^2,
  \end{aligned}
 \end{equation}
 and
\begin{equation}\label{aux316}
 \begin{split}
  \rho_u \|(\nabla \bv)\bn\|_{L^2(\OGamma)}^2 \leq \rho_u \sum_{T \in \cT_h^\Gamma}\|\nabla \bv\|_{L^2(T)}^2 \lesssim
  \rho_u \sum_{T \in \cT_h^\Gamma} h_T^2\|\nabla q\|_{L^2(T)}^2 \lesssim \|q\|_{1,\,h}^2.
 \end{split}
\end{equation}
From \eqref{aux302}, \eqref{aux315}, and \eqref{aux316} we conclude $\|\bv\|_A \lesssim \|q\|_{1,\,h}$, and using \eqref{mainest3} we  get
\[
 \|\bv\|_A \lesssim   \|q\|_{1,\,\textnormal{reg}}+s_h(q,q)^\frac12.
\]
Combining this with \eqref{E8} completes the proof.
\end{proof}

\begin{remark}\rm
There are several places in the proof of Theorem~\ref{lemma2}, where we use that for $q \in Q_h$ its local gradient $\nabla q_{|T}$ is a constant vector (i.e., $k=1$). In \eqref{est1} using  pressure elements for $k\ge2$ leads to
terms with higher order derivatives, and applying a FE inverse inequality to handle these does not lead to a satisfactory  bound. For the argument in \eqref{E5_0} it is also necessary that the pressure gradient is piecewise constant. A generalization of the analysis presented above, that applies to  $k \ge 1 $ and also addresses the effects of geometric  errors is a topic of current research.
\end{remark}

\section{Error bounds}\label{sec:error}

We consider $k=1$, i.e. the trace Taylor--Hood $\vect P_2$--$P_1$ pair.
Based on the stability result an error analysis for the \emph{consistent} variant, cf. Remark~\ref{remconsistent}, can be derived with standard arguments, which follow the standard line of combining stability, consistency and interpolation results; see, e.g., \cite{brezzi2012mixed} for general treatment of saddle-point problems, and more specific analysis of the surface Stokes problem in~\cite{olshanskii2018finite}.
We outline the arguments below and skip most of the details that can be found elsewhere.
First, we introduce the following bilinear form (with $\tilde A_h(\cdot,\cdot)$ as in Remark~\ref{remconsistent}):
\begin{equation} \label{defk}
 \bA_h\big((\bu,p),(\bv,q)\big)\coloneqq  \tilde A_h(\bu,\bv)+ b_T(\bv,p) +b_T(\bu,q) - s_h(p,q).
\end{equation}
The discrete problem \eqref{discrete}, with $A_h(\cdot,\cdot)$ replaced by $\tilde A_h(\cdot,\cdot)$, then has the compact representation: Determine $(\bu_h,p_h) \in \bU_h \times Q_h$ such that
\begin{equation} \label{discrete1}
 \bA_h\big((\bu_h,p_p),(\bv_h,q_h)\big)=(\blf,\bv_h) - (g,q_h) \quad \text{for all}~~(\bv_h,q_h) \in \bU_h \times Q_h.
\end{equation}
This discrete problem has a unique solution, which is denoted by $(\bu_h,p_h)$. Due to consistency, the true solution of \eqref{strongform-1}--\eqref{strongform-2} $(\bu_T^\ast, p^\ast)$ satisfies
\[
  \bA_h\big((\bu_T^\ast,p^\ast),(\bv,q)\big) = (\blf,\bv) - (g,q).  \label{consis1}
\]
Thus we obtain the Galerkin orthogonality relation,
\[
 \bA_h\big((\bu_T^\ast - \bu_h,p^\ast- p_h)),(\bv_h,q_h)\big) =0~~\text{for all}~~(\bv_h,q_h) \in \bU_h \times Q_h.
\]
The inf-sup stability~\eqref{LBB}, coercivity of $\tilde A_h(\cdot,\cdot)$, and the Galerkin orthogonality results yield the usual bound for the discretization error in terms of an approximation error in our problem-dependent norms,
\begin{equation} \label{optimal_err}
\|\bu_T^\ast -\bu_h\|_A + \|p^\ast - p_h \|_h \lesssim \inf_{(\bv_h,q_h)\in \bU_h \times Q_h}\big(\|\bu_T^\ast -\bv_h\|_A + \|p^\ast - q_h \|_h\big).
\end{equation}
Employing standard interpolation estimates for $\vect P_2$ and $P_1$ trace finite elements~(see, e.g., \cite{reusken2015analysis, olshanskii2016trace}) and assuming the necessary smoothness of $(\bu_T^\ast,p^\ast)$, we get an estimate for the right-hand side of~\eqref{optimal_err}:
\begin{equation} \label{interp}
\inf_{(\bv_h,q_h)\in \bU_h \times Q_h}\big(\|\bu_T^\ast -\bv_h\|_A + \|p^\ast - q_h \|_h\big) \lesssim h^2 (\|\bu_T^\ast\|_{H^3(\Gamma)} +\|p^\ast\|_{H^2(\Gamma)}).
\end{equation}
For the $O(h^2)$ bound in \eqref{interp} to hold, it is sufficient to assume the following bounds for the parameters entering the definition of $\|\cdot\|_A$ and $\|\cdot\|_h$ norms:
$\tau\lesssim h^{-2}$, $\rho_u\lesssim h^{-1}$, and $\rho_p \lesssim h$. Combining these restrictions with those needed for stability, we conclude that, for the parameters satisfying~\eqref{parameters}, equations~\eqref{optimal_err} and \eqref{interp} yield the optimal error estimate in the problem-dependent norm
\begin{equation} \label{mainerrbound}
  \|\bu_T^\ast -\bu_h\|_A + \|p^\ast - p_h \|_h \lesssim h^2  (\|\bu_T^\ast\|_{H^3(\Gamma)} +\|p^\ast\|_{H^2(\Gamma)}).
\end{equation}
The definition of $\|\cdot\|_A$ and $\|\cdot\|_h$  norms  and \eqref{mainerrbound} together give
\begin{align*}
 \|\bu_T^\ast  -(\bu_h)_T\|_{H^1(\Gamma)}  + \|p^\ast  - p_h \|_{L^2(\Gamma)}  &\lesssim h^2 (\|\bu_T^\ast\|_{H^3(\Gamma)} +\|p^\ast\|_{H^2(\Gamma)}), \\
  \|\bu_h\cdot\bn \|_{L^2(\Gamma)}  &\lesssim h^3 (\|\bu_T^\ast\|_{H^3(\Gamma)} +\|p^\ast\|_{H^2(\Gamma)}).
\end{align*}
Using $H^2$-regularity from Lemma~\ref{Lemma:regularity} a duality argument can be applied, cf.~\cite{olshanskii2018finite}. It results in the optimal error bound in the surface $L^2$ norm for velocity:
\begin{equation} \label{L2error}
\|\bu_T^\ast  - (\bu_h)_T\|_{L^2(\Gamma)} \lesssim h^{3}(\|\bu_T^\ast\|_{H^3(\Gamma)} +\|p^\ast\|_{H^2(\Gamma)}).
\end{equation}

\section{Numerical experiments} \label{sec:num}

\let\oldtabular\tabular
\renewcommand{\tabular}[1][1.5]{\def\arraystretch{#1}\oldtabular}

We present several numerical examples, which illustrate the analysis of this paper. In particular, we calculate the optimal constant $c_0$ from the key inf-sup condition \eqref{LBB} and demonstrate that it is indeed bounded independent of $h$ and the position of $\Gamma$ against the background mesh. We show that our analysis is sharp in the sense that without the normal volume stabilization we do not have discrete inf-sup stability. Numerical results will be also presented that demonstrate the optimal convergence of the consistent variant of the method.

\subsection{Setup}
We choose $\Gamma$ to be either the unit sphere or a torus, $\Gamma = \sphere$ or $\Gamma = \tor$. The corresponding level-set functions are given by
\begin{equation}\label{lchoice}
	\lsphere(\vect x) \coloneqq  \|\vect x\|^2 - 1, \quad
	\ltor(\vect x) \coloneqq  (\|\vect x\|^2 + R^2 - r^2)^2 - 4\,R^2\,(x^2+y^2),
\end{equation}
with $R \coloneqq  1$ and $r \coloneqq  0.2$. The computational  domain is $\Omega \coloneqq  (-5/3,\,5/3)^3$ such that $\Gamma\subset\Omega$ for both examples. In all the experiments we set $\alpha = 1$ in \eqref{strongform-1}.
To build the initial triangulation $\T_{h_0}$ we divide $\Omega$ into $2^3$ cubes and further tessellate each cube into 6 tetrahedra; Thus we have $h_0 = \frac53$. The mesh is gradually refined towards the surface, and $\ell \in \mathbb{N}$ denotes the level of refinement, with the mesh size $h_\ell = \frac53\,2^{-\ell}$; see Figure~\ref{fig:gamma} for an illustration of the bulk meshes and the induced mesh on the embedded surface for three consecutive refinement levels.

\begin{figure}[ht!]
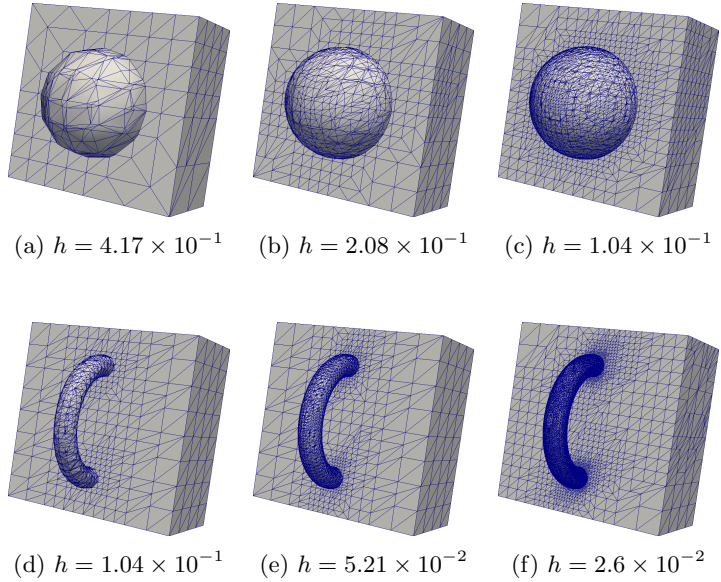

	\centering
	\begin{subfigure}{.25\linewidth}
		\centering
		\includegraphicsw[.9]{{lvl1.cropped}.png}
		\caption{$h = 4.17\times10^{-1}$}
	\end{subfigure}%
	\begin{subfigure}{.25\linewidth}
		\centering
		\includegraphicsw[.9]{{lvl2.cropped}.png}
		\caption{$h = 2.08\times10^{-1}$}
	\end{subfigure}%
	\begin{subfigure}{.25\linewidth}
		\centering
		\includegraphicsw[.9]{{lvl3.cropped}.png}
		\caption{$h = 1.04\times10^{-1}$}
	\end{subfigure}
	\par\bigskip
	\begin{subfigure}{.25\linewidth}
		\centering
		\includegraphicsw[.9]{{tor_lvl3.cropped}.png}
		\caption{$h = 1.04\times10^{-1}$}
	\end{subfigure}%
	\begin{subfigure}{.25\linewidth}
		\centering
		\includegraphicsw[.9]{{tor_lvl4.cropped}.png}
		\caption{$h = 5.21\times10^{-2}$}
	\end{subfigure}%
	\begin{subfigure}{.25\linewidth}
		\centering
		\includegraphicsw[.9]{{tor_lvl5.cropped}.png}
		\caption{$h = 2.6\times10^{-2}$}
	\end{subfigure}
	\caption{Cross  section of the bulk mesh and $\Gamma_h$ for~$\sphere$ (top) and $\tor$ (bottom), $\ell \in \{\,2, 3, 4\,\}$ and $\ell \in \{\,4, 5, 6\,\}$, respectively} \label{fig:gamma}
	\vskip -.5cm	
\end{figure}

In the following sections we numerically compute optimal  inf-sup stability constants for $\vect P_2$--$P_1$ and $\vect P_1$--$P_1$ trace FEM and show convergence results for $\vect P_2$--$P_1$ trace FEM.

\subsection{Discrete inf-sup stability}

Let $n \coloneqq \mbox{dim}(\bU_h)$ and $m \coloneqq \mbox{dim}(Q_h)$ be the number of velocity and pressure degrees of freedom (d.o.f.).
The velocity stiffness matrix~$\vect A \in \mathbb R^{n \times n}$, the divergence matrix~$\vect B \in \mathbb R^{m \times n}$, and pressure stabilization matrix~$\vect C_\star \in \mathbb R^{m \times m}$ satisfy
\begin{equation*}
\begin{split}
    &  \bar{\vect v}^T\vect A \bar{\vect u}=
   \left\{ \begin{split}
              A_h(\bu_h, \bv_h)  &\quad \text{for inconsistent method,} \\
              \widetilde{A}_h(\bu_h, \bv_h)  &\quad \text{for consistent method,}
           \end{split}
   \right.\quad
    \\[1ex]
      &  \bar{\vect q}^T\vect B \bar{\vect v}=b_T(q_h,\bv_h),
    \\
     &  \bar{\vect q}^T\vect C_\star \bar{\vect p}=s(q_h,p_h)
\end{split}
\end{equation*}
for all $\bu_h,\bv_h\in\bU_h$ and all $p_h, q_h\in Q_h$. Here $\bar{\vect u},\bar{\vect v} \in \mathbb R^{n}$ and $\bar{\vect p},\bar{\vect q} \in \mathbb R^{m}$ are the velocity and pressure vectors of coefficients for the corresponding finite element functions in a standard nodal basis.
The trace finite element method~\eqref{discrete}  results in the linear system
\begin{equation}\label{spsystem}
	\begin{bmatrix}
		\vect A & \phantom{-}\vect B^T \\
		\vect B & -\vect C_\star \\
	\end{bmatrix}
	\begin{bmatrix} \bar{\vect u} \\ \bar{\vect p} \end{bmatrix}
	=
	\begin{bmatrix} \phantom{-}\bar{\vect f} \\ -\bar{\vect g} \end{bmatrix},
\end{equation}
 where $\bar{\vect f} \in \mathbb R^{n}$ and $\bar{\vect g} \in \mathbb R^{m}$ are load vectors for the momentum and continuity equations right-hand sides, respectively.
 The pressure Schur complement matrix of \eqref{spsystem} is given by
\begin{equation}\label{schur}
\vect S_\star = \vect B\,\vect A^{-1}\,\vect B^{T} + \vect C_\star.
\end{equation}

We consider three different matrices~$\vect C_\star$ corresponding to three different choices of the stabilization form $s_h$:
\begin{enumerate}
	\item Volume normal stabilization given in~\eqref{forms}. We write $\vect C_\star = \vect C_n$ in this case;
	\item The full gradient stabilization given by the bilinear form in \eqref{FullStab} (Brezzi--Pitk\"{a}ranta type stabilization). We write~$\vect C_\star = \vect C_{\text{full}}$;
	\item No stabilization, i.e.~$\vect C_\star = \vect C_0 \coloneqq  \vect 0$.
\end{enumerate}
We recall that the method analyzed in this paper corresponds to  $\vect C_\star = \vect C_n$.

The surface pressure mass matrix~$\vect M_0 \in \mathbb R^{m \times m}$ is defined through $\bar{\vect q}^T\vect M_0\,\bar{\vect p} = \int_{\Gamma} p_h\,q_h \diff{s}$. We also need auxiliary matrices
\begin{equation}
	\vect M_n \coloneqq  \vect M_0 + \vect C_n,\quad
	\vect M_{\text{full}} \coloneqq  \vect M_0 + \vect C_{\text{full}},
\end{equation}
which correspond to the natural norms used in the pressure space, e.g., $\vect M_n$ corresponds to $\|\cdot\|_h$ from \eqref{norms}.

We are interested in the generalized eigenvalue problem
\begin{equation}\label{problem}
	\vect S_\star\,\bar{\vect y} = \lambda\,\vect M_\star\,\bar{\vect y},
\end{equation}
where ``$\star$'' stands for ``$0$,'' ``$n$,'' or ``full.''  We use notation ~$0 = \lambda_1 < \lambda_2 \le \dots \le \lambda_{m} = O(1)$ for the  generalized  eigenvalues of~\eqref{problem}.
Straightforward calculations show (cf., e.g., \cite[Lemma~5.9]{olshanskii2014iterative}) that for $\vect S_\star=\vect S_n$, $\vect M_\star=\vect M_n$  the smallest non-zero eigenvalue equals $c_0^2$,
\[\lambda_2=c_0^2,\]
where $c_0^2$ is the optimal constant from the inf-sup stability condition \eqref{LBB} (with each term squared).

Assembling the Schur complement  matrix $\vect S_\star$ becomes prohibitively expensive even for rather small mesh sizes, since one needs to calculate $\vect A^{-1}$.
One possible solution is to write~\eqref{problem} in the mixed form:
\begin{equation} \begin{split}
 & \begin{bmatrix}
		\vect A & \phantom{-}\vect B^T \\
		\vect B & -\vect C_\star \\
	\end{bmatrix}
\begin{bmatrix} \bar{\vect x} \\ \bar{\vect y} \end{bmatrix}
=
-\lambda\:\,
\begin{bmatrix}
	\,\,\,\vect 0\:\,\, & \\
	& \vect M_\star
	\end{bmatrix}
\begin{bmatrix} \bar{\vect x} \\ \bar{\vect y} \end{bmatrix}, \\
\text{leading to}\quad &
%
\underbrace{\begin{bmatrix}
		\vect A & \phantom{-}\vect B^T \\
		\vect B & -\vect C_\star \\
	\end{bmatrix}}_{\mathcal A_\star}
\begin{bmatrix} \bar{\vect x} \\ \bar{\vect y} \end{bmatrix}
=
-\lambda_{\epsilon}
\underbrace{\begin{bmatrix}
	\epsilon\,\vect A & \\
	& \vect M_\star
	\end{bmatrix}}_{\mathcal M^\epsilon_\star}
\begin{bmatrix} \bar{\vect x} \\ \bar{\vect y} \end{bmatrix},
\label{problem_pert}
\end{split}
\end{equation}
with $0 < \epsilon \ll 1$.
Here we introduced an $\epsilon$ perturbation to the right-hand matrix to make it  Hermitian positive definite. In this form, the problem is suitable for any standard generalized eigenvalue solver that operates with sparse Hermitian  matrices. The spectrum of the perturbed problem consists of two sets of eigenvalues, $\mbox{sp}([\mathcal M^\epsilon_\star]^{-1}\mathcal A_\star)=\Lambda_\epsilon\cup \Lambda_{\epsilon^{-1}}$. The eigenvalues from the first set converge to those of \eqref{problem}:
\begin{equation*}
	\lambda_{\epsilon} = \lambda + o(1)\text{ as }\epsilon \rightarrow 0,\quad\text{with }\lambda_\epsilon \in \Lambda_\epsilon\text{ and }\lambda \in \mbox{sp}(\vect M_\star^{-1}\vect S_\star).
\end{equation*}
For the eigenvalues in the other set we have $-\lambda_\epsilon= O(\epsilon^{-1})$, $\lambda_\epsilon\in \Lambda_{\epsilon^{-1}}$.
This makes it straightforward for $\epsilon\ll 1$ to identify the eigenvalues we are interested in. To simplify the computation further, we replace the $(1, 1)$-block of~$\mathcal M^\epsilon_\star$ by $\epsilon\,\vect I$.

To check that our computations are stable with respect to small $\epsilon$ and yield consistent results, we solve~\eqref{problem_pert} for~$\epsilon = 10^{-5}$ and~$\epsilon = 10^{-6}$. It turns out that we obtain very close results. Furthermore, for the coarse mesh levels, when solving~\eqref{problem} is feasible, we also check that the dense solver for~\eqref{problem} and the iterative one for~\eqref{problem_pert} with $\epsilon = 10^{-6}$ give eigenvalues that coincide at least up to the  first five significant digits.

Tables~\ref{tab:p2p1} and~\ref{tab:p1p1} report $\lambda_2$ (i.e. the inf-sup stability constant) and $\lambda_m$ (i.e. the maximum eigenvalue so that $\lambda_m/\lambda_2$ defines the effective condition number) for the following methods:  1) consistent $\vect P_2$--$P_1$ trace finite element method (studied in this paper); 2)  $\vect P_1$--$P_1$ trace finite element method from~\cite{olshanskii2018finite}. For both discretizations we solve the eigenvalue problem \eqref{problem_pert} with different matrices $\vect C_\star$ which correspond to three choices of pressure stabilization (see above).

\begin{table}[ht]
	\centering\footnotesize
	\caption{Extreme eigenvalues of~\eqref{problem} for consistent $\vect P_2$--$P_1$ trace finite element method, $\tau = h^{-2}$, $\rho_u = h^{-1}$, $\rho_p = h$}
	\label{tab:p2p1}
	\begin{subtable}{1.\linewidth}
		\centering
		\begin{tabular}[1.2]{|c|c|c|c|c|c|c|c|c|}
			\hline
			\multicolumn{9}{|c|}{$\Gamma = \sphere$} \\
			\hline
			\multirow{2}{*}{$h$} & \multirow{2}{*}{$n$} & \multirow{2}{*}{$m$} & \multicolumn{2}{c|}{$\vect S_0$} & \multicolumn{2}{c|}{$\vect S_n$} & \multicolumn{2}{c|}{$\vect S_{\text{full}}$} \\
			\cline{4-9}
			& & & $\lambda_2$ & $\lambda_{m}$ & $\lambda_2$ & $\lambda_{m}$ & $\lambda_2$ & $\lambda_{m}$ \\
			\hline
			\input{tab/sphere_2_P2P1_consistent_table.tex}
		\end{tabular}
	\end{subtable}%

	\begin{subtable}{1.\linewidth}
		\centering
		\begin{tabular}[1.2]{|c|c|c|c|c|c|c|c|c|}
			\hline
			\multicolumn{9}{|c|}{$\Gamma = \tor$} \\
			\hline
			\multirow{2}{*}{$h$} & \multirow{2}{*}{$n$} & \multirow{2}{*}{$m$} & \multicolumn{2}{c|}{$\vect S_0$} & \multicolumn{2}{c|}{$\vect S_n$} & \multicolumn{2}{c|}{$\vect S_{\text{full}}$} \\
			\cline{4-9}
			& & & $\lambda_2$ & $\lambda_{m}$ & $\lambda_2$ & $\lambda_{m}$ & $\lambda_2$ & $\lambda_{m}$ \\
			\hline
			\input{tab/torus_P2P1_consistent_table.tex}
		\end{tabular}
	\end{subtable}%
\end{table}

\begin{table}[ht]
	\centering\footnotesize
	\caption{Extreme eigenvalues of~\eqref{problem} for $\vect P_1$--$P_1$ trace finite element method, $\tau = h^{-2}$, $\rho_u = h$, $\rho_p = h$}
	\label{tab:p1p1}
	\begin{subtable}{1.\linewidth}
		\centering
		\begin{tabular}[1.2]{|c|c|c|c|c|c|c|c|c|}
			\hline
			\multicolumn{9}{|c|}{$\Gamma = \sphere$} \\
			\hline
			\multirow{2}{*}{$h$} & \multirow{2}{*}{$n$} & \multirow{2}{*}{$m$} & \multicolumn{2}{c|}{$\vect S_0$} & \multicolumn{2}{c|}{$\vect S_n$} & \multicolumn{2}{c|}{$\vect S_{\text{full}}$} \\
			\cline{4-9}
			& & & $\lambda_2$ & $\lambda_{m}$ & $\lambda_2$ & $\lambda_{m}$ & $\lambda_2$ & $\lambda_{m}$ \\
			\hline
			\input{tab/sphere_2_P1P1_table.tex}
		\end{tabular}
	\end{subtable}%

	\begin{subtable}{1.\linewidth}
		\centering
		\begin{tabular}[1.2]{|c|c|c|c|c|c|c|c|c|}
			\hline
			\multicolumn{9}{|c|}{$\Gamma = \tor$} \\
			\hline
			\multirow{2}{*}{$h$} & \multirow{2}{*}{$n$} & \multirow{2}{*}{$m$} & \multicolumn{2}{c|}{$\vect S_0$} & \multicolumn{2}{c|}{$\vect S_n$} & \multicolumn{2}{c|}{$\vect S_{\text{full}}$} \\
			\cline{4-9}
			& & & $\lambda_2$ & $\lambda_{m}$ & $\lambda_2$ & $\lambda_{m}$ & $\lambda_2$ & $\lambda_{m}$ \\
			\hline
			\input{tab/torus_P1P1_table.tex}
		\end{tabular}
	\end{subtable}
\end{table}
For experiments with  $\vect P_2$--$P_1$  elements we choose parameters satisfying \eqref{parameters}. In particular, we set $\rho_u=h^{-1}$ which is the upper extreme for admissible parameters, since for smaller $\rho_u$, the stability constant $c_0$ from \eqref{LBB} can only increase. We otherwise set $\rho_u =h$ for~$\vect P_1$--$P_1$ elements (which was the choice in~\cite{olshanskii2018finite}); if the resulting method is inf-sup unstable for $\rho_u \simeq h$, it has the same property also  for larger $\rho_u$.  In practice, to avoid extra dissipation, one would like to choose $\rho_u$ small but sufficient to control the condition number of $\bA$ ($\rho_u \simeq h$ suits this purpose). We covered the larger range of $\rho_u$ by our analysis, since this may be needed if the $\vect P_2$--$P_1$  trace FEM is extended to fluid problems on deformable surfaces.

From Table~\ref{tab:p2p1}, which shows results for $\vect P_2$--$P_1$ trace elements, we see that for $\vect C_0$ (no pressure stabilization) $\lambda_2$ tends to zero with mesh refinement, which indicates that the discretization   is not inf-sup stable. The normal gradient stabilization matrix~$\vect C_n$ is sufficient for the inf-sup stability, $\lambda_2$ is uniformly bounded from below, which confirms the main result of this paper.
Of course, including the  full pressure gradient term also leads to a stable method, but in this case, the method has consistency errors that are suboptimal.

For the two cases $\Gamma = \sphere$  and $\Gamma = \tor$ the behavior is essentially the same.
From Table~\ref{tab:p1p1} we see that only full gradient stabilization matrix~$\vect C_{\text{full}}$ guarantees inf-sup stability of~$\vect P_1$--$P_1$ trace elements, which is different to the situation with $\vect P_2$--$P_1$ trace elements.

Next, we illustrate our claim that the optimal inf-sup stability constant $c_0$ in \eqref{LBB} is  uniformly bounded with respect to the position of $\Gamma$ in  the background mesh.  To this end, we introduce a set of translated surfaces
\begin{equation}\label{shift}
	\Gamma \mapsto \Gamma + \alpha\,\vect s,
\end{equation}
with some $\alpha \in \mathbb R$ and $\vect s \in \mathbb R^3$, $\|\vect s\| = 1$; see Figure~\ref{fig:shift}.

\begin{figure}[ht]
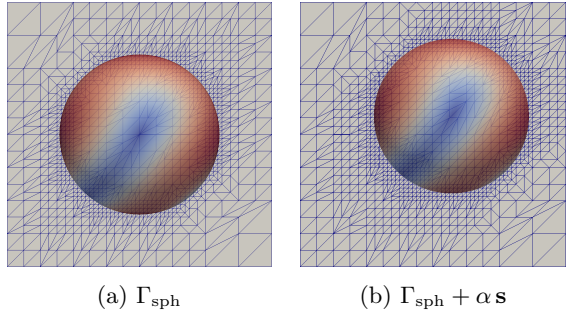

	\centering
	\begin{subfigure}{.2\linewidth}\end{subfigure}%
	\begin{subfigure}{.3\linewidth}
		\centering
		\includegraphicsw[.9]{{shift_0.0.cropped}.png}
		\caption{$\sphere$}
	\end{subfigure}%
	\begin{subfigure}{.3\linewidth}
		\centering
		\includegraphicsw[.9]{{shift_0.4.cropped}.png}
		\caption{$\sphere + \alpha\,\vect s$}
	\end{subfigure}%
	\begin{subfigure}{.2\linewidth}\end{subfigure}%
	\caption{Unit sphere (left) and the shifted unit sphere (right). Here $\vect s = (1, 1, 1)^T/\sqrt{3}$, $\alpha = 0.4$, and $h = 5.21\times10^{-2}$} \label{fig:shift}
	\vskip -.5cm	
\end{figure}

We repeat eigenvalue computations for the consistent $\vect P_2$--$P_1$ trace finite element method, with a fixed mesh size~$h = 1.04\times10^{-1}$ and a varying  translation parameter~$\alpha$ in~\eqref{shift}. Results are reported in Table~\ref{tab:p1p1_shift_h=0.104167}.

\begin{table}[ht]
	\centering
	\caption{Extreme eigenvalues of~\eqref{problem} for perturbed surface $\Gamma + \alpha\,\vect s$ for consistent $\vect P_2$--$P_1$ trace finite element method, $\tau = h^{-2}$, $\rho_u = h^{-1}$, $\rho_p = h$. Here $\vect s = (1, 1, 1)^T/\sqrt{3}$, $h = 1.04\times10^{-1}$}
	\label{tab:p1p1_shift_h=0.104167}
	\footnotesize
	\begin{subtable}{1.\linewidth}
		\centering
		\begin{tabular}[1.2]{|c|c|c|c|c|}
			\hline
			\multirow{2}{*}{Surface} & \multicolumn{2}{c|}{$\vect S_0$} & \multicolumn{2}{c|}{$\vect S_n$} \\
			\cline{2-5}
			& $\lambda_2$ & $\lambda_{m}$ & $\lambda_2$ & $\lambda_{m}$ \\
			\hline
			\input{tab/sphere_2_P2P1_consistent_h=0.104167_table.tex}
		\end{tabular}
	\end{subtable}
	\begin{subtable}{1.\linewidth}
		\centering
		\begin{tabular}[1.2]{|c|c|c|c|c|}
			\hline
			\multirow{2}{*}{Surface} & \multicolumn{2}{c|}{$\vect S_0$} & \multicolumn{2}{c|}{$\vect S_n$} \\
			\cline{2-5}
			& $\lambda_2$ & $\lambda_{m}$ & $\lambda_2$ & $\lambda_{m}$ \\
			\hline
			\input{tab/torus_P2P1_consistent_h=0.104167_table.tex}
		\end{tabular}
	\end{subtable}
\end{table}

The results in Table~\ref{tab:p1p1_shift_h=0.104167} confirm the robustness of the inf-sup stability constant with respect to the position of $\Gamma$ for the method with   normal gradient stabilization.

\subsection{Convergence results}

We set $\Gamma = \sphere$ and define
\begin{equation}\label{exact_soln_2}
	\tilde{\vect u}(x, y, z) \coloneqq  (-z^2, y, x)^T, \quad
	\tilde p(x, y, z) \coloneqq  x\,y^2 + z,\quad
	\vect u \coloneqq  \vect P\,\tilde{\vect u}^e, \quad
	p \coloneqq  \tilde p^e.
\end{equation}
Thus we have~$\int_\Gamma p \diff{\vect x} = 0$, $p \equiv p^e$, $\vect u \equiv \vect u^e$ in $\mathcal O(\Gamma)$, and $\vect u$ is a tangential vector field.

Note that for our choice of~$\lsphere$ in~\eqref{lchoice} we have~$\vect n = \nabla \lsphere/\|\nabla \lsphere\| \equiv \vect n^e$ in $\mathcal O(\Gamma)$. We use $P_2$ nodal interpolant~$I^2_h(\cdot)$ defined in~$\T_h^\Gamma$ to approximate the level-set function in our implementation.
Note that for the case of the sphere~$\lsphere \in P_2(\mathbb R^3)$ implies~$I^2_h(\lsphere) = \lsphere$, and so the normal vector and related operators $\vect P$, $E_s$, $\vect H$, and $\nabla_\Gamma$ are computed exactly.

To approximate the domain of integration, we use a sufficiently refined piecewise planar approximation of~$\Gamma$,
\begin{equation} \label{surf_int}
	\Gamma_h \coloneqq  \{ \vect x \in \mathbb R^3\::\:\big(I^1_h(\phi)\big)(\vect x) = 0\}.
\end{equation}
For the integration, we use~$\Gamma_{h/m}$ with $m \simeq h^{-1}$ so that we have an $O(h^4)$-accurate approximation of~$\Gamma$.

\begin{figure}[h]
	\centering
	\begin{subfigure}{.1\linewidth}\end{subfigure}%
	\begin{subfigure}{.4\linewidth}
		\centering
		\includegraphicsw[.8]{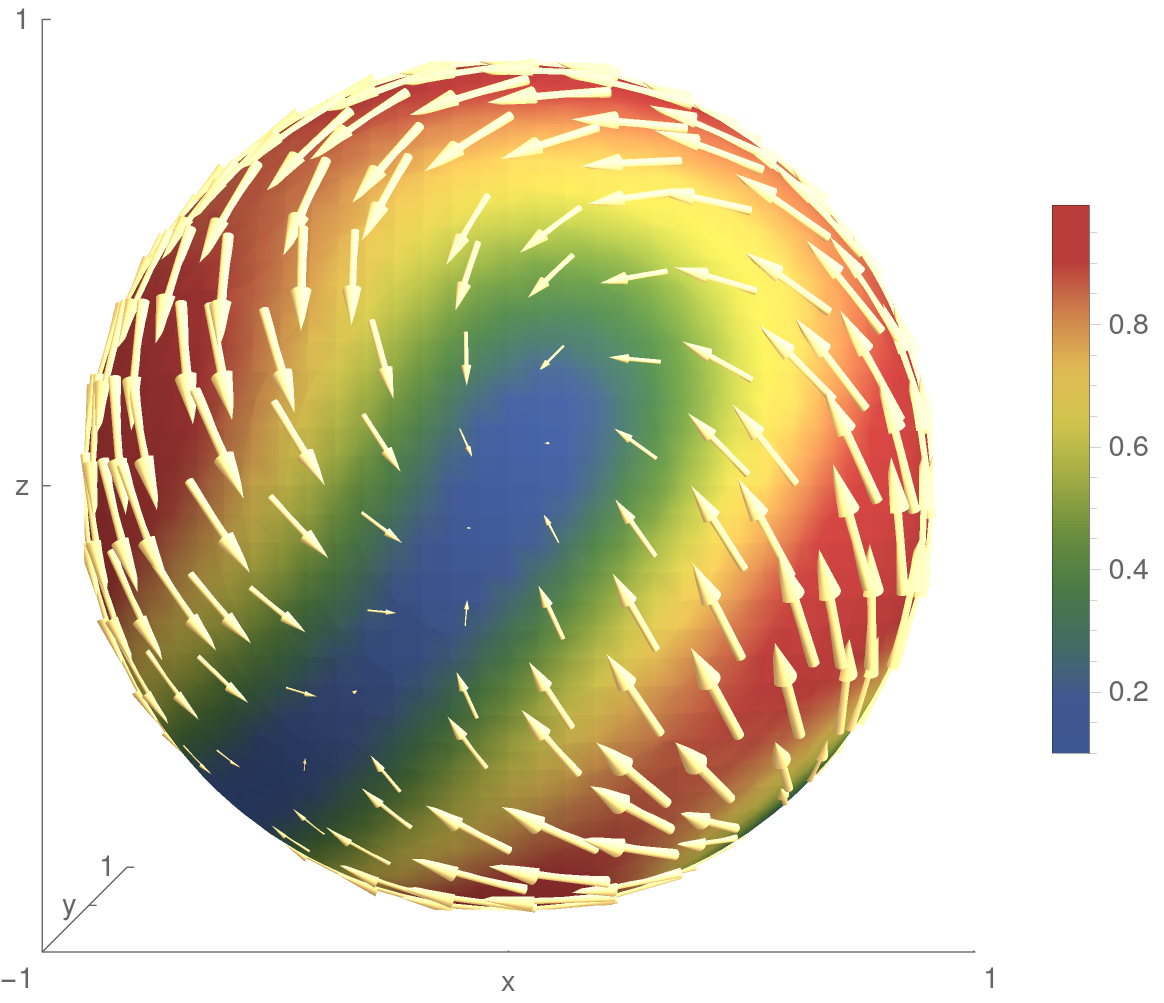}
	\end{subfigure}%
	\begin{subfigure}{.4\linewidth}
		\centering
		\includegraphicsw[.8]{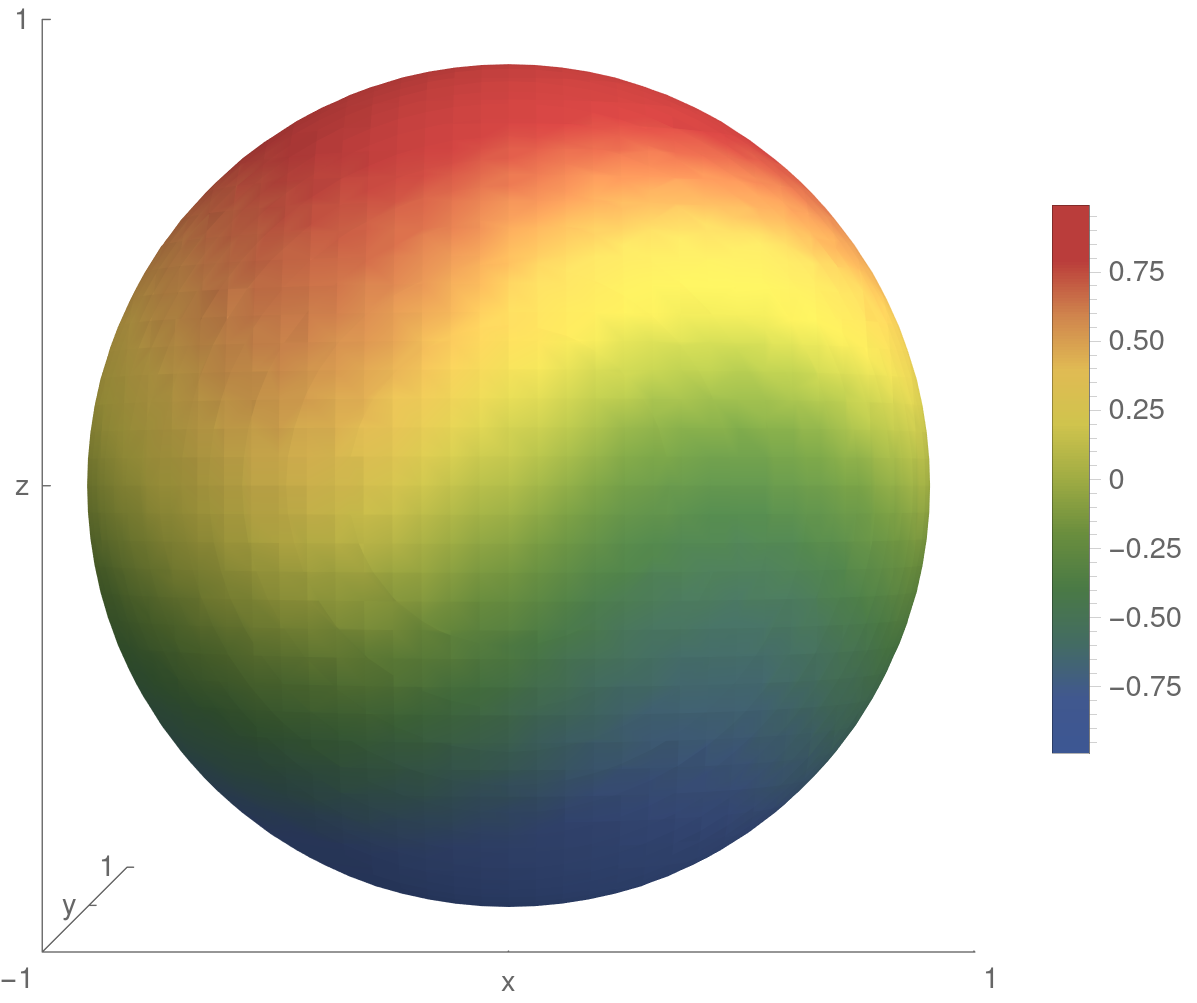}
	\end{subfigure}%
	\begin{subfigure}{.1\linewidth}\end{subfigure}%
	\caption{Exact velocity solution (left) and pressure solution (right) as in~\eqref{exact_soln_2}} \label{fig:soln}	
	\vskip -.5cm	
\end{figure}

The resulting linear system~\eqref{spsystem} is solved with MINRES as an outer solver and the  block-diagonal preconditioner as defined in~\cite[p.~19]{olshanskii2018finite}.

\begin{table}[ht]
	\centering\footnotesize
	\caption{Convergence results for consistent~$\vect P_2$--$P_1$ formulation, $\tau = h^{-2}$, $\rho_u = h^{-1}$, $\rho_p = h$, and $\vect C_{\star} = \vect C_n$}
	\label{tab:p2p1_cons_h^-1}
	\begin{subtable}{1.\linewidth}\centering
		\begin{tabular}[1.2]{|c||c|c||c|c||c||c|}
			\hline
			$h$ & $\|\vect u - \vect u_h\|_{\HOne}$ & Order & $\|\vect u - \vect u_h\|_{\LTwo}$ & Order & $\|p - p_h\|_{\LTwo}$ & Order \\
			\hline
            \input{tab/sphere_2_P2P1_patchnormals=0_shift=0.0_form=consistent_m=1_rhopfac=1_rhoppow=1_rhoufac=1_rhoupow=-1_conv_table.tex}
		\end{tabular}
	\end{subtable}

	\begin{subtable}{1.\linewidth}\centering
		\begin{tabular}[1.2]{|c||c|c||c||c|}
			\hline
			$h$ & $\| \vect u_h\cdot\vect n \|_{\LTwo}$ & Order & Outer iterations & Residual norm \\
			\hline
			\input{tab/sphere_2_P2P1_patchnormals=0_shift=0.0_form=consistent_m=1_rhopfac=1_rhoppow=1_rhoufac=1_rhoupow=-1_iters_table.tex}
		\end{tabular}
	\end{subtable}
\end{table}

\begin{table}[ht]
	\centering\footnotesize
	\caption{Convergence results for consistent~$\vect P_2$--$P_1$ formulation, $\tau = h^{-2}$, $\rho_u = h^{-1}$, $\rho_p = h$, and $\vect C_{\star} = \vect C_{\text{full}}$}
	\label{tab:p2p1_cons_h^-1_full}
	\begin{subtable}{1.\linewidth}\centering
		\begin{tabular}[1.2]{|c||c|c||c|c||c||c|}
			\hline
			$h$ & $\|\vect u - \vect u_h\|_{\HOne}$ & Order & $\|\vect u - \vect u_h\|_{\LTwo}$ & Order & $\|p - p_h\|_{\LTwo}$ & Order \\
			\hline
			\input{tab/sphere_2_P2P1_patchnormals=0_shift=0.0_form=consistent_m=12_stab=full_rhopfac=1_rhoppow=1_rhoufac=1_rhoupow=-1_conv_table.tex}
		\end{tabular}
	\end{subtable}

	\begin{subtable}{1.\linewidth}\centering
		\begin{tabular}[1.2]{|c||c|c||c||c|}
			\hline
			$h$ & $\| \vect u_h\cdot\vect n \|_{\LTwo}$ & Order & Outer iterations & Residual norm \\
			\hline
			\input{tab/sphere_2_P2P1_patchnormals=0_shift=0.0_form=consistent_m=12_stab=full_rhopfac=1_rhoppow=1_rhoufac=1_rhoupow=-1_iters_table.tex}
		\end{tabular}
	\end{subtable}
\end{table}
%

We first consider the  convergence rates of the  (consistent) $\vect P_2$--$P_1$ trace finite element given by \eqref{discrete}--\eqref{Aconsistent}. This corresponds to the volume normal derivative stabilization matrix~$\vect C_n$. Results are reported in Table~\ref{tab:p2p1_cons_h^-1}. From the table we see that the consistent formulation gives optimal convergence rates in all the norms as predicted by the error analysis in section~\ref{sec:error}.

Table~\ref{tab:p2p1_cons_h^-1_full} reports results of a further experiment in which  the volume normal derivative stabilization for the pressure is replaced by the \emph{full} gradient stabilization, i.e the stabilization matrix~$\vect C_{\rm full}$ is used. This option was discussed in Remark~\ref{rem:stab}. It is expected that this results in suboptimal convergence rates due to only $O(h^2)$ consistency of the added term. This is  what we see in Table~\ref{tab:p2p1_cons_h^-1_full}, which shows suboptimal rates in $\LTwoSpace$-velocity error norm.
In both cases reported in Tables~\ref{tab:p2p1_cons_h^-1} and~\ref{tab:p2p1_cons_h^-1_full}, the number of MINRES iterations is reasonable and stays uniformly bounded with respect to the variation of mesh parameter~$h$.

\section{Conclusions and outlook} \label{sec:concl}
 The paper presented the first stability analysis of a mixed pair of $\vect P_2$--$P_1$ finite elements for the discretization of a surface Stokes problem. For the finite element discretization of (Navier--)Stokes equations in Euclidean domains the lowest order stable Taylor--Hood element is one of the most popular methods.   The results of this paper prove stability and optimal order convergence for the trace variant of the  $\vect P_2$--$P_1$ Taylor--Hood element, i.e., this method is a good candidate for discretization of  fluid problems on manifolds.

 Numerical experiments (not shown here) indicate that the higher order trace Taylor--Hood pairs, combined with the volume normal derivative stabilization, are also stable. In a companion paper~\cite{JankuhnStokes2019} numerical results for higher order trace Taylor--Hood and a detailed comparison of the consistent and inconsistent variants are presented. That paper also addresses the issue of geometry approximation (using the parametric version of trace FEM) and the accuracy of the normal used in the penalty term. The latter issue was earlier analyzed  in~\cite{hansbo2016analysis,JankuhnVekorLaplace} for the vector Laplace surface problem.
 A rigorous error analysis including geometry errors, using the isoparametric trace FEM, will be studied by the authors in a forthcoming paper. 

\subsection*{Acknowledgement} M.O. and A.Z. were partially supported by NSF through the Division of Mathematical Sciences grant 1717516.

\bibliographystyle{siam}
\bibliography{main}{}

\end{document}

%% file: img/illustrProof1.pdf_tex
\begingroup%
  \makeatletter%
  \providecommand\color[2][]{%
    \errmessage{(Inkscape) Color is used for the text in Inkscape, but the package 'color.sty' is not loaded}%
    \renewcommand\color[2][]{}%
  }%
  \providecommand\transparent[1]{%
    \errmessage{(Inkscape) Transparency is used (non-zero) for the text in Inkscape, but the package 'transparent.sty' is not loaded}%
    \renewcommand\transparent[1]{}%
  }%
  \providecommand\rotatebox[2]{#2}%
  \newcommand*\fsize{\dimexpr\f@size pt\relax}%
  \newcommand*\lineheight[1]{\fontsize{\fsize}{#1\fsize}\selectfont}%
  \ifx\svgwidth\undefined%
    \setlength{\unitlength}{396.96154785bp}%
    \ifx\svgscale\undefined%
      \relax%
    \else%
      \setlength{\unitlength}{\unitlength * \real{\svgscale}}%
    \fi%
  \else%
    \setlength{\unitlength}{\svgwidth}%
  \fi%
  \global\let\svgwidth\undefined%
  \global\let\svgscale\undefined%
  \makeatother%
  \begin{picture}(1,0.78598079)%
    \lineheight{1}%
    \setlength\tabcolsep{0pt}%
    \put(0,0){\includegraphics[width=\unitlength,page=1]{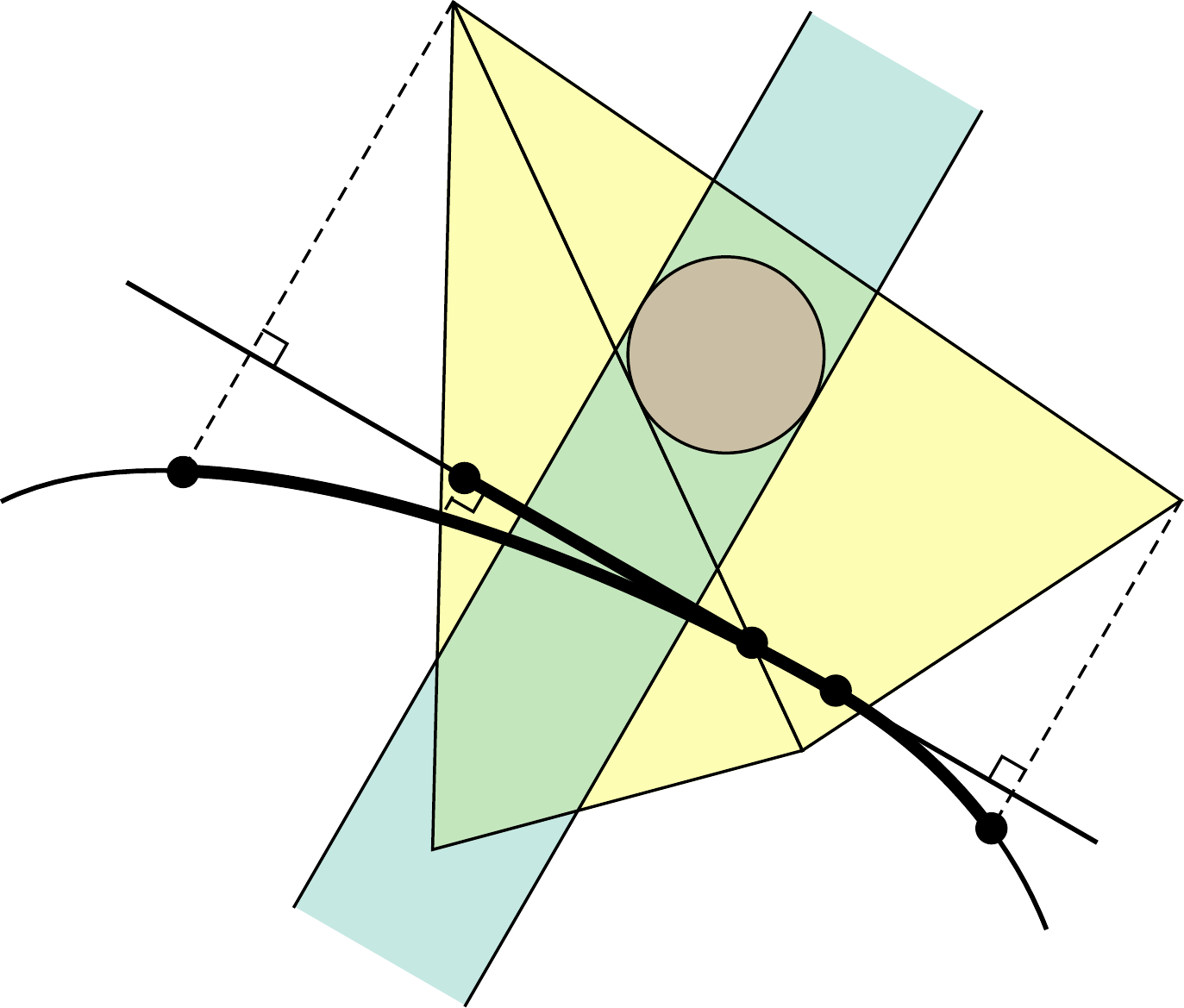}}%
    \put(0.0524313,0.38){\color[rgb]{0,0,0}\makebox(0,0)[lt]{\lineheight{1.25}\smash{\begin{tabular}[t]{l}\rotatebox{10}{$\Gamma$}\end{tabular}}}}%
    \put(0.18,0.39){\color[rgb]{0,0,0}\makebox(0,0)[lt]{\lineheight{1.25}\smash{\begin{tabular}[t]{l}\rotatebox{-15}{$P_\Gamma(T)$}\end{tabular}}}}%
    \put(0.57,0.22){\color[rgb]{0,0,0}\makebox(0,0)[lt]{\lineheight{1.25}\smash{\begin{tabular}[t]{l}$T'$\end{tabular}}}}%
    \put(0.77,0.4){\color[rgb]{0,0,0}\makebox(0,0)[lt]{\lineheight{1.25}\smash{\begin{tabular}[t]{l}$T$\end{tabular}}}}%
    \put(0.41,0.465){\color[rgb]{0,0,0}\makebox(0,0)[lt]{\lineheight{1.25}\smash{\begin{tabular}[t]{l}\rotatebox{-25}{$S$}\end{tabular}}}}%
    \put(0.48,0.422){\color[rgb]{0,0,0}\makebox(0,0)[lt]{\lineheight{1.25}\smash{\begin{tabular}[t]{l}\rotatebox{-25}{$D_S$}\end{tabular}}}}%
    \put(0.68,0.72){\color[rgb]{0,0,0}\makebox(0,0)[lt]{\lineheight{1.25}\smash{\begin{tabular}[t]{l}$C_S$\end{tabular}}}}%
    \put(0.573,0.535){\color[rgb]{0,0,0}\makebox(0,0)[lt]{\lineheight{1.25}\smash{\begin{tabular}[t]{l}$B_T$\end{tabular}}}}%
    \put(0.69,0.3){\color[rgb]{0,0,0}\makebox(0,0)[lt]{\lineheight{1.25}\smash{\begin{tabular}[t]{l}$\vect x_0$\end{tabular}}}}%
    \put(0.124,0.63){\color[rgb]{0,0,0}\makebox(0,0)[lt]{\lineheight{1.25}\smash{\begin{tabular}[t]{l}\rotatebox{-29}{$\lPlane$}\end{tabular}}}}%
  \end{picture}%
\endgroup%

%% file: img/illustrBaseFace.pdf_tex
\begingroup%
  \makeatletter%
  \providecommand\color[2][]{%
    \errmessage{(Inkscape) Color is used for the text in Inkscape, but the package 'color.sty' is not loaded}%
    \renewcommand\color[2][]{}%
  }%
  \providecommand\transparent[1]{%
    \errmessage{(Inkscape) Transparency is used (non-zero) for the text in Inkscape, but the package 'transparent.sty' is not loaded}%
    \renewcommand\transparent[1]{}%
  }%
  \providecommand\rotatebox[2]{#2}%
  \newcommand*\fsize{\dimexpr\f@size pt\relax}%
  \newcommand*\lineheight[1]{\fontsize{\fsize}{#1\fsize}\selectfont}%
  \ifx\svgwidth\undefined%
    \setlength{\unitlength}{346.72833252bp}%
    \ifx\svgscale\undefined%
      \relax%
    \else%
      \setlength{\unitlength}{\unitlength * \real{\svgscale}}%
    \fi%
  \else%
    \setlength{\unitlength}{\svgwidth}%
  \fi%
  \global\let\svgwidth\undefined%
  \global\let\svgscale\undefined%
  \makeatother%
  \begin{picture}(1,0.90951771)%
    \lineheight{1}%
    \setlength\tabcolsep{0pt}%
    \put(0,0){\includegraphics[width=\unitlength,page=1]{illustrBaseFace.pdf}}%
    \put(0.19233028,0.19281946){\color[rgb]{0,0,0}\makebox(0,0)[lt]{\lineheight{1.25}\smash{\begin{tabular}[t]{l}$\delta$\end{tabular}}}}%
    \put(0.717995,0.21634327){\color[rgb]{0,0,0}\makebox(0,0)[lt]{\lineheight{1.25}\smash{\begin{tabular}[t]{l}\rotatebox{45}{$\delta$}\end{tabular}}}}%
    \put(0.42352346,-0.01145202){\color[rgb]{0,0,0}\makebox(0,0)[lt]{\lineheight{1.25}\smash{\begin{tabular}[t]{l}$\hat F_b$\end{tabular}}}}%
    \put(0.59684668,0.09255725){\color[rgb]{0,0,0}\makebox(0,0)[lt]{\lineheight{1.25}\smash{\begin{tabular}[t]{l}$\hat{\vect n}_{\hat F_b}$\end{tabular}}}}%
    \put(0.32541513,0.50146814){\color[rgb]{0,0,0}\makebox(0,0)[lt]{\lineheight{1.25}\smash{\begin{tabular}[t]{l}$\hat{\vect n}(\vect y)$\end{tabular}}}}%
    \put(0.28329797,0.42554195){\color[rgb]{0,0,0}\makebox(0,0)[lt]{\lineheight{1.25}\smash{\begin{tabular}[t]{l}$\vect y$\end{tabular}}}}%
    \put(0.40207475,0.22558286){\color[rgb]{0,0,0}\makebox(0,0)[lt]{\lineheight{1.25}\smash{\begin{tabular}[t]{l}$\hat T_\delta$\end{tabular}}}}%
    \put(0.04409651,0.41089889){\color[rgb]{0,0,0}\makebox(0,0)[lt]{\lineheight{1.25}\smash{\begin{tabular}[t]{l}$\hat \Gamma_{\hat T}$\end{tabular}}}}%
    \put(0.14887034,0.00878801){\color[rgb]{0,0,0}\makebox(0,0)[lt]{\lineheight{1.25}\smash{\begin{tabular}[t]{l}0\end{tabular}}}}%
    \put(0.98125758,0.00612932){\color[rgb]{0,0,0}\makebox(0,0)[lt]{\lineheight{1.25}\smash{\begin{tabular}[t]{l}1\end{tabular}}}}%
    \put(0.108343,0.87078923){\color[rgb]{0,0,0}\makebox(0,0)[lt]{\lineheight{1.25}\smash{\begin{tabular}[t]{l}1\end{tabular}}}}%
  \end{picture}%
\endgroup%

%% file: tab/sphere_2_P2P1_consistent_table.tex
$8.33\times	10^{-1}$	&	$789$	&	$51$	&	$2.33\times	10^{-1}$	&	$1.07$	&	$6.3\times	10^{-1}$	&	$1.$	&	$8.81\times	10^{-1}$	&	$1.$	\\ \hline
$4.17\times	10^{-1}$	&	$3276$	&	$190$	&	$4.72\times	10^{-2}$	&	$6.97\times	10^{-1}$	&	$5.29\times	10^{-1}$	&	$1.$	&	$7.64\times	10^{-1}$	&	$1.$	\\ \hline
$2.08\times	10^{-1}$	&	$11718$	&	$664$	&	$7.93\times	10^{-2}$	&	$6.7\times	10^{-1}$	&	$5.09\times	10^{-1}$	&	$1.$	&	$6.39\times	10^{-1}$	&	$1.$	\\ \hline
$1.04\times	10^{-1}$	&	$48762$	&	$2764$	&	$3.71\times	10^{-2}$	&	$6.69\times	10^{-1}$	&	$5.03\times	10^{-1}$	&	$1.$	&	$5.73\times	10^{-1}$	&	$1.$	\\ \hline
$5.21\times	10^{-2}$	&	$193086$	&	$10912$	&	$1.81\times	10^{-3}$	&	$6.68\times	10^{-1}$	&	$4.98\times	10^{-1}$	&	$1.$	&	$5.36\times	10^{-1}$	&	$1.$	\\ \hline
$2.6\times	10^{-2}$	&	$775998$	&	$43864$	&	$6.65\times	10^{-4}$	&	$6.65\times	10^{-1}$	&	$4.92\times	10^{-1}$	&	$1.$	&	$5.17\times	10^{-1}$	&	$1.$	\\ \hline

%% file: tab/torus_P2P1_consistent_table.tex
$2.08\times	10^{-1}$	&	$5580$	&	$324$	&	$2.15\times	10^{-1}$	&	$9.56\times	10^{-1}$	&	$3.12\times	10^{-1}$	&	$1.$	&	$3.4\times	10^{-1}$	&	$1.$	\\ \hline
$1.04\times	10^{-1}$	&	$28116$	&	$1580$	&	$1.59\times	10^{-2}$	&	$7.6\times	10^{-1}$	&	$3.21\times	10^{-1}$	&	$1.$	&	$3.35\times	10^{-1}$	&	$1.$	\\ \hline
$5.21\times	10^{-2}$	&	$116592$	&	$6568$	&	$1.31\times	10^{-3}$	&	$7.48\times	10^{-1}$	&	$3.21\times	10^{-1}$	&	$1.$	&	$3.26\times	10^{-1}$	&	$1.$	\\ \hline
$2.6\times	10^{-2}$	&	$477708$	&	$26936$	&	$1.9\times	10^{-4}$	&	$7.42\times	10^{-1}$	&	$3.2\times	10^{-1}$	&	$1.$	&	$3.22\times	10^{-1}$	&	$1.$	\\ \hline

%% file: tab/sphere_2_P1P1_table.tex
$8.33\times	10^{-1}$	&	$153$	&	$51$	&	$1.32\times	10^{-2}$	&	$1.42$	&	$7.48\times	10^{-1}$	&	$1.13$	&	$9.58\times	10^{-1}$	&	$1.06$	\\ \hline
$4.17\times	10^{-1}$	&	$570$	&	$190$	&	$5.12\times	10^{-3}$	&	$1.04$	&	$5.77\times	10^{-1}$	&	$1.$	&	$8.54\times	10^{-1}$	&	$1.$	\\ \hline
$2.08\times	10^{-1}$	&	$1992$	&	$664$	&	$4.4\times	10^{-3}$	&	$7.93\times	10^{-1}$	&	$3.87\times	10^{-1}$	&	$1.$	&	$6.71\times	10^{-1}$	&	$1.$	\\ \hline
$1.04\times	10^{-1}$	&	$8292$	&	$2764$	&	$2.01\times	10^{-3}$	&	$7.79\times	10^{-1}$	&	$2.19\times	10^{-1}$	&	$1.$	&	$5.82\times	10^{-1}$	&	$1.$	\\ \hline
$5.21\times	10^{-2}$	&	$32736$	&	$10912$	&	$6.04\times	10^{-5}$	&	$9.81\times	10^{-1}$	&	$1.17\times	10^{-1}$	&	$1.$	&	$5.37\times	10^{-1}$	&	$1.$	\\ \hline
$2.6\times	10^{-2}$	&	$131592$	&	$43864$	&	$3.53\times	10^{-5}$	&	$8.67\times	10^{-1}$	&	$5.72\times	10^{-2}$	&	$1.$	&	$5.16\times	10^{-1}$	&	$1.$	\\ \hline
$1.3\times	10^{-2}$	&	$525864$	&	$175288$	&	$2.16\times	10^{-6}$	&	$7.34\times	10^{-1}$	&	$2.84\times	10^{-2}$	&	$1.$	&	$5.04\times	10^{-1}$	&	$1.$	\\ \hline

%% file: tab/torus_P1P1_table.tex
$2.08\times	10^{-1}$	&	$972$	&	$324$	&	$5.04\times	10^{-2}$	&	$4.93$	&	$2.84\times	10^{-1}$	&	$1.35$	&	$3.64\times	10^{-1}$	&	$1.19$	\\ \hline
$1.04\times	10^{-1}$	&	$4740$	&	$1580$	&	$2.99\times	10^{-3}$	&	$3.83$	&	$1.58\times	10^{-1}$	&	$1.02$	&	$3.35\times	10^{-1}$	&	$1.01$	\\ \hline
$5.21\times	10^{-2}$	&	$19704$	&	$6568$	&	$1.11\times	10^{-3}$	&	$5.45$	&	$7.73\times	10^{-2}$	&	$1.01$	&	$3.25\times	10^{-1}$	&	$1.$	\\ \hline
$2.6\times	10^{-2}$	&	$80808$	&	$26936$	&	$1.2\times	10^{-4}$	&	$5.42$	&	$3.07\times	10^{-2}$	&	$1.01$	&	$3.21\times	10^{-1}$	&	$1.$	\\ \hline
$1.3\times	10^{-2}$	&	$327036$	&	$109012$	&	$1.77\times	10^{-5}$	&	$5.23$	&	$1.18\times	10^{-2}$	&	$1.01$	&	$3.16\times	10^{-1}$	&	$1.$	\\ \hline

%% file: tab/sphere_2_P2P1_consistent_h=0.104167_table.tex
$\sphere + 0.0\,\vect s$	&	$3.71\times	10^{-2}$	&	$6.69\times	10^{-1}$	&	$5.03\times	10^{-1}$	&	$1.$	\\ \hline
$\sphere + 0.1\,\vect s$	&	$1.31\times	10^{-3}$	&	$6.87\times	10^{-1}$	&	$5.03\times	10^{-1}$	&	$1.$	\\ \hline
$\sphere + 0.2\,\vect s$	&	$1.25\times	10^{-3}$	&	$6.70\times	10^{-1}$	&	$5.03\times	10^{-1}$	&	$1.$	\\ \hline
$\sphere + 0.3\,\vect s$	&	$1.04\times	10^{-2}$	&	$6.72\times	10^{-1}$	&	$5.03\times	10^{-1}$	&	$1.$	\\ \hline
$\sphere + 0.4\,\vect s$	&	$5.32\times	10^{-4}$	&	$6.72\times	10^{-1}$	&	$5.03\times	10^{-1}$	&	$1.$	\\ \hline 

%% file: tab/torus_P2P1_consistent_h=0.104167_table.tex
$\tor + 0.00\,\vect s$	&	$1.59\times	10^{-2}$	&	$7.6\times	10^{-1}$	&	$3.21\times	10^{-1}$	&	$1.$	\\ \hline
$\tor + 0.05\,\vect s$	&	$9.20\times	10^{-3}$	&	$1.14$	&	$3.21\times	10^{-1}$	&	$1.$	\\ \hline
$\tor + 0.10\,\vect s$	&	$3.00\times	10^{-3}$	&	$1.91$	&	$3.19\times	10^{-1}$	&	$1.$	\\ \hline
$\tor + 0.15\,\vect s$	&	$8.67\times	10^{-3}$	&	$1.02$	&	$3.21\times	10^{-1}$	&	$1.$	\\ \hline
$\tor + 0.20\,\vect s$	&	$6.68\times	10^{-3}$	&	$3.04$	&	$3.21\times	10^{-1}$	&	$1.$	\\ \hline 

%% file: tab/sphere_2_P2P1_patchnormals=0_shift=0.0_form=consistent_m=1_rhopfac=1_rhoppow=1_rhoufac=1_rhoupow=-1_conv_table.tex
$8.3\times	10^{-1}$	&	$2.2$	&	$\text{}$	&	$6.4\times	10^{-1}$	&	$\text{}$	&	$7.4\times	10^{-1}$	&	$\text{}$	\\ \hline
$4.2\times	10^{-1}$	&	$3.8\times	10^{-1}$	&	$2.5$	&	$6.1\times	10^{-2}$	&	$3.4$	&	$1.2\times	10^{-1}$	&	$2.6$	\\ \hline
$2.1\times	10^{-1}$	&	$9.2\times	10^{-2}$	&	$2.1$	&	$5.8\times	10^{-3}$	&	$3.4$	&	$2.5\times	10^{-2}$	&	$2.2$	\\ \hline
$1.\times	10^{-1}$	&	$2.2\times	10^{-2}$	&	$2.1$	&	$5.6\times	10^{-4}$	&	$3.4$	&	$6.1\times	10^{-3}$	&	$2.1$	\\ \hline
$5.2\times	10^{-2}$	&	$5.3\times	10^{-3}$	&	$2.$	&	$5.2\times	10^{-5}$	&	$3.4$	&	$1.6\times	10^{-3}$	&	$1.9$	\\ \hline
$2.6\times	10^{-2}$	&	$1.3\times	10^{-3}$	&	$2.$	&	$5.2\times	10^{-6}$	&	$3.3$	&	$4.1\times	10^{-4}$	&	$2.$	\\ \hline
$1.3\times	10^{-2}$	&	$3.4\times	10^{-4}$	&	$2.$	&	$6.\times	10^{-7}$	&	$3.1$	&	$1.\times	10^{-4}$	&	$2.$	\\ \hline

%% file: tab/sphere_2_P2P1_patchnormals=0_shift=0.0_form=consistent_m=1_rhopfac=1_rhoppow=1_rhoufac=1_rhoupow=-1_iters_table.tex
$8.33\times	10^{-1}$	&	$4.5\times	10^{-1}$	&	$\text{}$	&	$26$	&	$2.6\times	10^{-9}$	\\ \hline
$4.17\times	10^{-1}$	&	$5.3\times	10^{-2}$	&	$3.1$	&	$33$	&	$5.1\times	10^{-9}$	\\ \hline
$2.08\times	10^{-1}$	&	$4.9\times	10^{-3}$	&	$3.4$	&	$31$	&	$6.\times	10^{-9}$	\\ \hline
$1.04\times	10^{-1}$	&	$5.\times	10^{-4}$	&	$3.3$	&	$27$	&	$7.3\times	10^{-9}$	\\ \hline
$5.21\times	10^{-2}$	&	$4.9\times	10^{-5}$	&	$3.4$	&	$25$	&	$6.4\times	10^{-9}$	\\ \hline
$2.6\times	10^{-2}$	&	$5.\times	10^{-6}$	&	$3.3$	&	$26$	&	$4.3\times	10^{-9}$	\\ \hline
$1.3\times	10^{-2}$	&	$5.8\times	10^{-7}$	&	$3.1$	&	$34$	&	$7.8\times	10^{-9}$	\\ \hline

%% file: tab/sphere_2_P2P1_patchnormals=0_shift=0.0_form=consistent_m=12_stab=full_rhopfac=1_rhoppow=1_rhoufac=1_rhoupow=-1_conv_table.tex
$8.3\times	10^{-1}$	&	$1.6$	&	$\text{}$	&	$7.8\times	10^{-1}$	&	$\text{}$	&	$1.3$	&	$\text{}$	\\ \hline
$4.2\times	10^{-1}$	&	$6.9\times	10^{-1}$	&	$1.2$	&	$3.9\times	10^{-1}$	&	$1.$	&	$8.1\times	10^{-1}$	&	$6.3\times	10^{-1}$	\\ \hline
$2.1\times	10^{-1}$	&	$2.4\times	10^{-1}$	&	$1.5$	&	$1.3\times	10^{-1}$	&	$1.6$	&	$3.1\times	10^{-1}$	&	$1.4$	\\ \hline
$1.\times	10^{-1}$	&	$8.1\times	10^{-2}$	&	$1.6$	&	$3.6\times	10^{-2}$	&	$1.8$	&	$1.1\times	10^{-1}$	&	$1.5$	\\ \hline
$5.2\times	10^{-2}$	&	$2.4\times	10^{-2}$	&	$1.8$	&	$9.5\times	10^{-3}$	&	$1.9$	&	$3.2\times	10^{-2}$	&	$1.7$	\\ \hline
$2.6\times	10^{-2}$	&	$6.5\times	10^{-3}$	&	$1.9$	&	$2.4\times	10^{-3}$	&	$2.$	&	$8.8\times	10^{-3}$	&	$1.9$	\\ \hline
$1.3\times	10^{-2}$	&	$1.8\times	10^{-3}$	&	$1.8$	&	$6.1\times	10^{-4}$	&	$2.$	&	$2.5\times	10^{-3}$	&	$1.8$	\\ \hline

%% file: tab/sphere_2_P2P1_patchnormals=0_shift=0.0_form=consistent_m=12_stab=full_rhopfac=1_rhoppow=1_rhoufac=1_rhoupow=-1_iters_table.tex
$8.33\times	10^{-1}$	&	$3.5\times	10^{-1}$	&	$\text{}$	&	$15$	&	$1.3\times	10^{-9}$	\\ \hline
$4.17\times	10^{-1}$	&	$5.4\times	10^{-2}$	&	$2.7$	&	$21$	&	$3.5\times	10^{-9}$	\\ \hline
$2.08\times	10^{-1}$	&	$4.9\times	10^{-3}$	&	$3.4$	&	$27$	&	$5.5\times	10^{-9}$	\\ \hline
$1.04\times	10^{-1}$	&	$5.\times	10^{-4}$	&	$3.3$	&	$29$	&	$5.8\times	10^{-9}$	\\ \hline
$5.21\times	10^{-2}$	&	$4.9\times	10^{-5}$	&	$3.4$	&	$29$	&	$5.1\times	10^{-9}$	\\ \hline
$2.6\times	10^{-2}$	&	$5.\times	10^{-6}$	&	$3.3$	&	$28$	&	$8.4\times	10^{-9}$	\\ \hline
$1.3\times	10^{-2}$	&	$5.9\times	10^{-7}$	&	$3.1$	&	$34$	&	$9.1\times	10^{-9}$	\\ \hline